\documentclass[a4paper]{article}
\usepackage[utf8]{inputenc}
\usepackage[T1]{fontenc}
\usepackage{indentfirst}
\usepackage[a4paper,top=2.54cm,bottom=2.0cm,left=3.0cm,right=3.54cm]{geometry}

\usepackage[round,comma]{natbib}

\usepackage{hyperref}
\usepackage{float}

\usepackage{amsmath}
\usepackage{amssymb}
\usepackage{amsthm}

\usepackage{algorithm}
\usepackage[noend]{algpseudocode}
\algrenewcommand\algorithmicindent{0.75em} 

\usepackage{tikz}
\usetikzlibrary{arrows.meta}
\tikzset{
    circ/.style={draw, circle,inner sep=0pt,minimum size=8mm, font=\scriptsize},
    triangle/.tip={Computer Modern Rightarrow[open,angle=120:3pt]}
}

\usepackage{caption}
\newtheorem{theorem}{Theorem}
\newtheorem{proposition}{Proposition}
\newtheorem{definition}{Definition}
\newtheorem{remark}{Remark}

\newtheorem{assumption}{Assumption}

\newtheorem{example}[theorem]{Example}

\def\P{{\mathbb P}}
\newcommand{\E}{\mathbb{E}}
\newcommand{\R}{\mathbb{R}}
\newcommand{\Z}{\mathbb{Z}}
\newcommand{\N}{\mathbb{N}}
\renewcommand{\H}{\mathcal{H}}

\newcommand{\A}{\mathcal{A}}

\newcommand{\mydef}{:=}

\title{\textbf{\Large{A new look at perfect simulation for chains with infinite memory}}}

\author{Emilio De Santis\footnote{Sapienza University of Rome, Department of Mathematics Piazzale Aldo Moro, 5, 00185, Rome, Italy
E-mail: desantis@mat.uniroma1.it} \ and Kádmo Laxa\footnote{Faculdade de Filosofia, Ciências e Letras de Ribeirão Preto, Universidade de São Paulo, Av. Bandeirantes, 3900,
Ribeirão Preto-SP, 14040-901, Brazil.
E-mail: kadmo.laxa@usp.br} \ and Eva Löcherbach\footnote{CMAP, Ecole Polytechnique, Institut Polytechnique de Paris, 91120 Palaiseau, France. E-mail: eva.loecherbach@polytechnique.edu}}

\date{\today}

\begin{document}

\maketitle


\begin{abstract}
In this article we introduce two new perfect simulation algorithms for chains with infinite memory. Both algorithms belong to the coupling of past procedures. The novelty of our approach is that it allows to include unknown states to the possible past symbols such that we can also deal with sparsely distributed past dependencies. In our first algorithm, spontaneous occurrence of symbols is possible. This means that there is a positive probability that the chain chooses the next symbol independently of the past. Our second algorithm deals with the case in which spontaneous occurrence of symbols is not possible. Chains with infinite memory are discrete-time stochastic processes in which the distribution of the next symbol depends on all past symbols. These transition probabilities are described by a probability kernel. Our results give conditions on the way the dependency of the transition kernel on long past strings decays, guaranteeing that our algorithms stop after a finite number of steps almost surely.  Strengthening these conditions, we show that the mean number of steps of our algorithms is finite. We discuss the consequence of having a coupling from the past algorithm with such properties and we present examples in which our results can be applied while others result in the literature cannot be applied.


\end{abstract}

\section{Introduction}
In this article, we introduce two new perfect simulation algorithms for chains with infinite memory. Chains with infinite memory, also known in the literature as \textit{chains of infinite order}, are discrete-time stochastic processes in which the distribution of the next symbol depends on all past symbols. They were first studied by \cite{completeconections} under the name \textit{chains with complete connections}. This class of stochastic processes generalizes the finite memory dependence of Markov chains, and the chains with memory of variable length introduced by \cite{rissanen} are a subclass of them.

By {\it perfect simulation algorithm} we mean a procedure to obtain a sample of a stochastic process in its stationary regime. To do so, we construct a convenient function that maps independent uniform random variables into the stationary sample of the process. Such a  function is called {\it coupling function}. The key idea is to explore the past to determine a {\it relevant subset} of all possible past sequences that suffices to sample from the stationary process within a prescribed finite time window. 

Commonly used procedures are the \textit{coupling from the past} algorithms which have been introduced in the seminal work of \cite{cftp} for Markov chains.  \cite{cff} extends the ideas of this algorithm to introduce a perfect simulation algorithm for chains with infinite memory. This subject was further explored by other articles, which present different perfect simulation algorithms for chains with infinite memory in the coupling from the past framework (see, for example, \cite{gallo}, \cite{desantispiccioni}, \cite{gallogarcia},\cite{gallotakahashi} and \cite{ciftp}).


We propose two perfect simulation algorithms for chains with infinite memory that belong to the coupling of past procedures. The novelty of our approach is that it allows to include unknown states to the possible past symbols such that we can also deal with sparsely distributed past dependencies. In our first algorithm, Algorithm \ref{algo:ps}, spontaneous occurrence of symbols is possible. This means that there is a positive probability that the chain chooses the next symbol independently of the past. The procedure proposed by Algorithm \ref{algo:ps} goes backward in time searching for spontaneously generated symbols. Once a spontaneous symbol is obtained, we go forward in time, updating the symbols that are not spontaneous. This procedure is repeated until the symbols of interest are obtained. Each spontaneous symbol is an information about the past history that works as a source for updating the following symbols.

Algorithm \ref{algo:ps2} deals with the case in which spontaneous occurrence of symbols is not possible. In this case, the above regenerative construction is not possible any more. 
Instead, we follow the original idea proposed by \cite{cftp}. We simultaneously construct the process starting from all possible past strings of a certain length. When all these simultaneous processes hit a single symbol, we can assign this symbol as the original process value at this precise time. The symbols generated by this simultaneous construction of the process with different pasts are the source we need to replicate the procedure introduced by Algorithm \ref{algo:ps}.

Let us now informally present our results. Theorems \ref{teo:cftp} and \ref{coro:cftp} show that Algorithms \ref{algo:ps} and \ref{algo:ps2} belong to the family of \textit{coupling from the past} algorithms. We then give conditions on the way the dependency of the transition kernel on long pasts decays guaranteeing that Algorithms \ref{algo:ps} and \ref{algo:ps2} stop after a finite number of steps almost surely.  Strengthening these conditions, we show that the mean number of steps of Algorithm \ref{algo:ps} and \ref{algo:ps2} is finite. This is the content of Theorems \ref{teo:infinitepath}, \ref{teo:renewal}, \ref{teo:algops2_1} and \ref{teo:algops2_2}.

The existence of perfect simulation algorithms that stop after a finite number of steps almost surely, which is guaranteed by Theorems \ref{teo:infinitepath}, \ref{teo:renewal}, \ref{teo:algops2_1} and \ref{teo:algops2_2},  has mathematical consequences which are well known. The process generated by each of our perfect simulation algorithms is in fact a sample of the unique stationary measure of the chain. This measure satisfies a loss of memory property, and the chain admits the concentration of measure phenomenon. Moreover, there exists a regeneration scheme for the chain. These results are summarized in Propositions \ref{prop:consequences}, \ref{prop:consequences2} and \ref{prop:gallotakahashi}.

Our construction is original and, as far as we know, the results we prove are more general than those presented in the literature, since we allow unknown states within the possible past strings such that we may also deal with sparsely distributed past dependencies. In particular, we present examples of chains in which our results guarantee that our perfect simulation algorithm stops in a finite time almost surely, while other results in the literature cannot guarantee this (see Sections \ref{sec:examples} and \ref{sec:dicussion}).

Our article is organized as follows. In Section \ref{sec:algonotationresults} we present our algorithms, present the basic notation, and state the main results. In Section \ref{sec:examples} we present some examples illustrating the conditions in which our algorithms can be used. In Section \ref{sec:proofteocftp} we prove Theorems \ref{teo:cftp} and \ref{coro:cftp}. In Section \ref{sec:proofsalgo1} we prove Theorems \ref{teo:infinitepath}, \ref{teo:renewal}, \ref{teo:algops2_1} and \ref{teo:algops2_2}. In Section \ref{sec:dicussion} we compare our results with other results in the literature and make some final observations.

\section{The algorithms, basic notation and main results} \label{sec:algonotationresults}

Let $\Z=\{\ldots,-2,-1,0,+1,2,\ldots\}$ be the set of integer numbers and $\N=\{0,1,2,\ldots\}$ be the set of natural numbers. Consider a set ${E}$ and let $ (a_n)_{n \in \Z } \in E^\Z $ be a sequence indexed by $ \Z $ with values in $E.$  For any $z\in \Z$ and $m\in \N$, let
$$
a_z^{z+m}\mydef (a_{z+m}, a_{z+m-1},\ldots, a_z) \in E^{m+1}
$$
and
$$
a_{-\infty}^{z} \mydef (a_{z}, a_{z-1},\ldots) \in E^{\N},
$$
reading a sequence from the present to the past. By convention, $a_z^{z-m}$ is the empty string if $ m > 0.$ For any $z_1,z_2\in \Z$ and $m_1,m_2\in \N$, we denote the concatenation of the strings $a_{z_1}^{z_1+m_1} \in E^{m_1+1}$ and $b_{z_2}^{z_2+m_2} \in E^{ m_2 +1 }$ as follows
$$
a_{z_1}^{z_1+m_1}b_{z_2}^{z_2+m_2} \mydef (a_{z_1+m_1}, a_{z_1+m_1-1},\ldots, a_{z_1}, b_{z_2+m_2}, b_{z_2+m_2-1},\ldots, b_{z_2} )\in E^{(m_1+1)+(m_2+1)} .
$$

\begin{definition} \label{def:kernel_and_adimissible_histories}
Let $\A$ be a countable and ordered set. A probability kernel is a function $p:\A \times \A^{\N} \to [0,1]$ such that
for any $a_{-\infty}^{-1} \in \A^{\N}$, $p(\cdot | a_{-\infty}^{-1}):\A \to [0,1]$ is a probability distribution. A process $(X_n)_{n \in \Z}$ is \textbf{compatible} with the probability kernel $p=(p(\cdot |a_{-\infty}^{-1}):a_{-\infty}^{-1} \in \A^{\N})$ if for all $g \in \A, $ for all $n,$ almost surely, 
$$ \P ( X_n = g | X^{n-1}_{- \infty} ) = p ( g |  X^{n-1}_{- \infty} ) .$$
\end{definition}

A probability kernel is extended to conditional distributions of finite sequences, given the past, by putting for any $n\geq 2$ and for any $a_{-\infty}^{-1} \in \A^{\N}$,
$$
p(a_{-n}^{-1}|a_{-\infty}^{-(n+1)})\mydef p(a_{-(n-1)}^{-1}|a_{-\infty}^{-n})p(a_{-n}|a_{-\infty}^{-(n+1)}).
$$
In what follows we will often restrict attention to infinite past sequences which are compatible with the transition kernel $p.$ To do so, we define the set of admissible histories as follows
\begin{equation} \label{eq:admissiblehistories}
\H := \left\{a_{-\infty}^{-1}\in \A^{\N}: \sup_{b_{-\infty}^{-(n+1)} \in \A^{\N}}p(a_{-n}^{-1}|b_{-\infty}^{-(n+1)})>0, \forall n\geq 1 \right\}.
\end{equation}

The set $\H$ contains all infinite strings $a_{-\infty}^{-1} \in \A^{\N}$ in which any finite substring $a_{-n}^{-1} \in \A^{n}$ is possible, for $n\geq1$. By possible we mean that there exists $b_{-\infty}^{-(n+1)} \in \A^{\N}$ such that $p(a_{-n}^{-1}|b_{-\infty}^{-(n+1)})>0$. $\H$ satisfies the \textit{remove invariance} property ($a_{-\infty}^{-1} \in \H$ implies that $a_{-\infty}^{-2} \in \H$) and the \textit{add invariance} property ($a_{-\infty}^{-1} \in \H$ and $p(g|a_{-\infty}^{-1})>0$ implies that $ga_{-\infty}^{-1} \in \H$).

\begin{proposition} \label{prop:admissiblehistories} If $(X_n)_{n\in \Z}$ is compatible with $p$, then $
\P(X_{-\infty}^{n} \in \H, \text{ for all } n\in \Z)=1.
$
    
\end{proposition}
\begin{proof}
Suppose that $(X_n)_{n\in \Z}$ is compatible with $p.$ For any $ m, n \in \Z,$ with $m\leq n$, put $E_m^n :=\{a_{-\infty}^{-1} X_{m}^{n} \not\in \H, \text{ for all }a_{-\infty}^{-1}\in \H\} .$ 
Assume that there exists $m,n\in \Z$ such that $\P (E_m^n) > 0. $ On $E_m^n, $ there exists $m\leq k= k ( \omega ) \leq n$ such that $p(X_k|X_{-\infty}^{k-1})=0.$ This implies that 
\begin{equation}\label{eq:firstfact}
p( X_m^{n}| X_{- \infty}^{m-1} ) = 0 \mbox{ on } E_m^n.
\end{equation} 
However, conditioning on $\sigma \{ X_k, k \le n-1 \},$  we have that 
\begin{multline*}
\P ( p( X_m^{n}| X_{- \infty}^{m-1} ) > 0 ) = \sum_{x_n \in \A} \E [ 1_{\{ p ( x_n |X_{- \infty}^{n-1} ) > 0 \}} 1_{\{ X_n = x_n \}}\prod_{j=m}^{n-1} 1_{\{ p ( X_j |X_{- \infty}^{j-1} ) > 0 \}}  ]  \\
=\sum_{x_n \in \A} \E [1_{\{ p ( x_n |X_{- \infty}^{n-1} ) > 0 \}}  p(x_n|X_{- \infty}^{n-1}) \prod_{j=m}^{n-1} 1_{\{ p ( X_j |X_{- \infty}^{j-1} ) > 0 \}}  ]  \\
= \E [ \sum_{x_n \in \A} p(x_n|X_{- \infty}^{n-1}) \prod_{j=m}^{n-1} 1_{\{ p ( X_j |X_{- \infty}^{j-1} ) > 0 \}}  ]  =  \E [  \prod_{j=m}^{n-1} 1_{\{ p ( X_j |X_{- \infty}^{j-1} ) > 0 \}}  ]  .
\end{multline*}
Iterating this argument, by conditioning next on $\sigma \{ X_k, k \le n-2 \},$ then on $\sigma \{ X_k, k \le n-3 \},$ and so on, and finally on $ \sigma \{ X_k, k \le m-1 \} , $
we see that 
$$ \P ( p( X_m^{n}| X_{- \infty}^{m-1} ) > 0 ) = 1 , $$ that is, almost surely, $p( X_m^{n}| X_{- \infty}^{m-1} ) > 0 ,$ which is in contradiction with (\ref{eq:firstfact}).    
\end{proof}

\begin{remark}
If we know that there exists a set $\tilde{\H}\subset \H$ such that if $(X_n)_{n\in \Z}$ is compatible with $p,$ then $\P(X_{-\infty}^{n} \in \tilde{\H}, \text{ for all } n\in \Z)=1$, then we can replace $\H$ by $\tilde{\H}$ in the sequel.
\end{remark}

The aim of this article is to construct perfect simulation algorithms to simulate $X_{-m}^{0}$ for a fixed $m\geq 0,$ where $(X_n)_{n \in \Z}$ is compatible with the probability kernel $p.$ We propose two algorithms to achieve this goal. The first algorithm, Algorithm \ref{algo:ps}, applies for kernels satisfying 
\begin{equation}\label{def:alpha}
\beta:=\sum_{g\in \A}\inf_{x^{-1}_{-\infty}\in \H}p(g|x^{-1}_{-\infty})>0.
\end{equation}
The second algorithm, Algorithm \ref{algo:ps2}, deals with kernels without this condition.
Before presenting the algorithms and summarizing their properties, we have to introduce the following notation.

For any $g\in \A$, let
$$
\alpha(g):=\inf_{a_{-\infty}^{-1} \in \H}p(g|a_{-\infty}^{-1}).
$$
Note that $\beta=\sum_{g\in \A}\alpha(g)$.
Moreover, for any $n\geq 1$, $g\in \A$ and $b_{-n}^{-1}\in \A^{n}$, let
\begin{equation}\label{def:alphaparticular}
\alpha(g|b_{-n}^{-1})\mydef\inf\{p(g|a_{-\infty}^{-1}): a_{-\infty}^{-1} \in \H, a_{-j}=b_{-j}, j=1,\ldots,n\} ,
\end{equation}
where we put $ \inf \emptyset := 0,$ 
and 
\begin{equation}\label{eq:beta}
\beta( b_{-n}^{-1})\mydef\sum_{g\in \A}\alpha(g|b_{-n}^{-1}).
\end{equation}

In (\ref{def:alphaparticular}) we take the infimum over the set of all histories that have the first $n$ symbols in the past fixed. We extend this notation to the case where the fixed part of the history is not necessarily given by the first symbols. To do so, we adjoin to our state space an extra symbol ``$*$'' that denotes ``unknown states'' and put $\A_*=\A\cup\{*\}$. 

For any $n\geq 1$, $g\in \A$ and $b_{-n}^{-1}\in \A_*^{n}$, let
\begin{equation}\label{def:alphageneral}
\alpha(g|b_{-n}^{-1})\mydef\inf\{p(g|a_{-\infty}^{-1}): a_{-\infty}^{-1} \in \H, a_{-j}=b_{-j}\;   \forall  j \in \{1, \ldots,n\} \mbox{ with }  b_{-j} \in \A\}
\end{equation}
and 
\begin{equation}\label{def:beta}
\beta(b_{-n}^{-1})\mydef\sum_{g\in \A}\alpha(g|b_{-n}^{-1}).
\end{equation}
Finally, let $\alpha(*)\mydef 1-\beta$ and $\alpha(*|b_{-n}^{-1}) \mydef 1-\beta( b_{-n}^{-1})$. If $ m > 0$, then $b_{z+m}^{z}$ is the empty string, and in this case we use the conventions $\alpha(g|b_{z+m}^{z})=\alpha(g)$ and $\beta(b_{z+m}^{z})=\beta$.

Note that \eqref{def:alphaparticular} is a particular case of \eqref{def:alphageneral} and is defined to improve the readability of the text.

\begin{remark}\label{rem:monotone} For any $b_{-\infty}^{-1} \in \A_*^{\N}$, it follows that for any $m>n\geq 0$ and for any $g\in \A$,
$$
\alpha(g|b_{-m}^{-1})\geq \alpha(g|b_{-n}^{-1}). 
$$
Moreover, if $\tilde{b}_{-\infty}^{-1} \in \A_*^{\N}$ satisfies $\tilde{b}_{-k} \in \{*,b_{-k}\}$, for any $k\geq 1$, then it follows that for any $m\geq 1$ and for any $g\in \A$,
$$
\alpha(g|b_{-m}^{-1})\geq \alpha(g|\tilde{b}_{-m}^{-1}).
$$
Note also that $\alpha(g|b^{-1}_{-n}*)=\alpha(g|b^{-1}_{-n})$, for any $g\in \A$ and $b^{-1}_{-n} \in \A_*^n$. In particular, $\alpha(g|*)=\alpha(g)$.
\end{remark}

\subsection{An introductory example}

Before presenting Algorithm \ref{algo:ps},
let us consider an example that will help us understand how the algorithm works. Let $(U_n)_{n\in \Z}$ be a sequence of i.i.d. uniform random variables in $[0,1)$. 

Consider the finite alphabet $\A=\{1,\ldots,m\}$. We wish to sample $ X_0 $ where $(X_n)_{n\in \Z}$ is compatible with $p$. To do so, we first consider a partition $ (I_0 ( k), 1 \le k \le m )$ of the interval $[0,\beta)$ containing $m$ right-open intervals of size $\alpha(1),\ldots, \alpha(m)$, respectively (see Figure \ref{fig:1}). 

\begin{figure}[H]
\centering
\begin{tikzpicture}[scale=7]
\draw[->, thick] (-0.1,0) -- (1.1,0);
\foreach \x/\xtext in {0/0,0.17/$\alpha(1)$,0.37/$\alpha(2)$,0.8/$\alpha(m)$,1/$\beta$}
\draw[thick] (\x,0.5pt) -- (\x,-0.5pt);
\draw (0.575,1pt) node[above] {$\ldots$};
\draw[|-|] (0,0.07) -- node[fill=white,inner sep=0.5mm,outer sep=0.5mm] {$\alpha(1)$} (0.17,0.07);
\draw[|-|] (0.17,0.07) -- node[fill=white,inner sep=0.5mm,outer sep=0.5mm] {$\alpha(2)$} (0.37,0.07);
\draw[|-|] (0.8,0.07) -- node[fill=white,inner sep=0.5mm,outer sep=0.5mm] {$\alpha(m)$} (1,0.07);
\draw (0,-0.06) node {$0$};
\draw (1,-0.06) node {$\beta$};
\end{tikzpicture}
\caption{}
\label{fig:1}
\end{figure}

We then sample the uniform random variable $U_0$. If $U_0<\beta$, we set $X^{(0)}_0=k$ if and only if $U_0 \in I_0 (k) ,$ that is, $ U_0$ belongs to the $k$-th sub-interval of $[0,\beta)$ which has size $\alpha(k),$ $ 1 \le k \le m.$ In the contrary case, we temporarily set $X^{(0)}_0=*$, meaning that we do not yet know the value of $X_0$. 

\begin{figure}[H]
\centering
\begin{tikzpicture}[scale=7]
\draw[->, thick] (-0.1,0) -- (1.1,0);
\foreach \x/\xtext in {0/0,0.2/$\beta$,1/$1$}
    \draw[thick] (\x,0.5pt) -- (\x,-0.5pt) node[below] {\xtext};
\draw[-, ultra thick, blue] (0,0) -- (0.2,0);
\draw[thick, red] (0.3,0.5pt) -- (0.3,-0.5pt) node[below] {$U_0$};
\draw (-0.25,0) node {$X^{(0)}_0=*$};
\end{tikzpicture}
\caption{}
\label{fig:2}
\end{figure}

If $X^{(0)}_0=k, $ for some $k\in \A$, we stop the algorithm and put $ X_0 = X^{(0)}_0. $ In the contrary case, we have to continue the algorithm. In our example, $U_0>\beta$ (see Figure \ref{fig:2}), such that we have to continue the algorithm.

So we look back trying to obtain the value of $X_{-1}$. We sample the random variable $U_{-1}$. If $U_{-1}<\beta$, we set $X^{(-1)}_{-1}=k$ if and only if $U_{-1} \in I_0 (k) ,$ 
$ 1 \le k \le m$. In the contrary case, we temporarily set $X^{(-1)}_{-1}=*$ and keep the previous value of the process at time $0,$ that is, we set $ X^{(-1)}_0 = *.$ In this example $U_{-1}>\beta$
(see Figure \ref{fig:3}) such that we have to go even further back into the past. 

\begin{figure}[H]
\centering
\begin{tikzpicture}[scale=7]
\draw[->, thick] (-0.1,0) -- (1.1,0);
\foreach \x/\xtext in {0/0,0.2/$\beta$,1/$1$}
    \draw[thick] (\x,0.5pt) -- (\x,-0.5pt) node[below] {\xtext};
\draw[-, ultra thick, blue] (0,0) -- (0.2,0);
\draw[thick, red] (0.3,0.5pt) -- (0.3,-0.5pt) node[below] {$U_0$};
\draw (-0.25,0) node {$X^{(-1)}_0=*$};
\end{tikzpicture}
\begin{tikzpicture}[scale=7]
\draw[->, thick] (-0.1,0) -- (1.1,0);
\foreach \x/\xtext in {0/0,0.2/$\beta$,1/$1$}
    \draw[thick] (\x,0.5pt) -- (\x,-0.5pt) node[below] {\xtext};
\draw[-, ultra thick, blue] (0,0) -- (0.2,0);
\draw[thick, red] (0.55,0.5pt) -- (0.55,-0.5pt) node[below] {$U_{-1}$};
\draw (-0.25,0) node {$X^{(-1)}_{-1}=*$};
\end{tikzpicture}
\caption{}
\label{fig:3}
\end{figure}

Since $X^{(-1)}_{-1}$ is unknown, we look back trying to obtain the value of $X_{-2}$. In this example $U_{-2}>\beta$ and then, $X^{(-2)}_{-2}=*$. Finally, in this example $U_{-3}<\beta$ and we denote $a_{-3}$ the index of the interval to which $U_{-3}$ belongs, i.e., $U_{-3} \in I_0 ( a_{-3} ) $ 
which has size $\alpha(a_{-3})$
 (see Figure \ref{fig:4}). Note that in  Figure \ref{fig:4} the values at times $ - 2, -1 $ and $0$ are not yet updated.

\begin{figure}[H]
\centering
\begin{tikzpicture}[scale=7]
\draw[->, thick] (-0.1,0) -- (1.1,0);
\foreach \x/\xtext in {0/0,0.2/$\beta$,1/$1$}
    \draw[thick] (\x,0.5pt) -- (\x,-0.5pt) node[below] {\xtext};
\draw[-, ultra thick, blue] (0,0) -- (0.2,0);
\draw[thick, red] (0.3,0.5pt) -- (0.3,-0.5pt) node[below] {$U_0$};
\draw (-0.25,0) node {$X^{(-2)}_0=*$};
\end{tikzpicture}
\begin{tikzpicture}[scale=7]
\draw[->, thick] (-0.1,0) -- (1.1,0);
\foreach \x/\xtext in {0/0,0.2/$\beta$,1/$1$}
    \draw[thick] (\x,0.5pt) -- (\x,-0.5pt) node[below] {\xtext};
\draw[-, ultra thick, blue] (0,0) -- (0.2,0);
\draw[thick, red] (0.55,0.5pt) -- (0.55,-0.5pt) node[below] {$U_{-1}$};
\draw (-0.25,0) node {$X^{(-2)}_{-1}=*$};
\end{tikzpicture}
\begin{tikzpicture}[scale=7]
\draw[->, thick] (-0.1,0) -- (1.1,0);
\foreach \x/\xtext in {0/0,0.2/$\beta$,1/$1$}
    \draw[thick] (\x,0.5pt) -- (\x,-0.5pt) node[below] {\xtext};
\draw[-, ultra thick, blue] (0,0) -- (0.2,0);
\draw[thick, red] (0.35,0.5pt) -- (0.35,-0.5pt) node[below] {$U_{-2}$};
\draw (-0.25,0) node {$X^{(-2)}_{-2}=*$};
\end{tikzpicture}
\begin{tikzpicture}[scale=7]
\draw[->, thick] (-0.1,0) -- (1.1,0);
\foreach \x/\xtext in {0/0,0.2/$\beta$,1/$1$}
    \draw[thick] (\x,0.5pt) -- (\x,-0.5pt) node[below] {\xtext};
\draw[-, ultra thick, blue] (0,0) -- (0.2,0);
\draw[thick, red] (0.1,0.5pt) -- (0.1,-0.5pt) node[below] {$U_{-3}$};
\draw (-0.25,0) node {$X^{(-3)}_{-3}=a_{-3}$};
\end{tikzpicture}
\caption{}
\label{fig:4}
\end{figure}

As we have obtained $X^{(-3)}_{-3} \in \A$, we now sequentially check if this is enough to obtain the values of $X^{(-3)}_{-2},X^{(-3)}_{-1}$ and $X^{(-3)}_{0}$. 
For $X^{(-3)}_{-2},$ we consider a partition $ (I_1 (k, U_{-3}), 1 \le k \le m )$ of the interval $[\beta, \beta( a_{-3}))$ containing $m$ right-open intervals of size $\alpha(k|a_{-3})-\alpha(k)$, respectively.
By Remark~\ref{rem:monotone}, note that we have that $\alpha(g|b)\geq \alpha(g)$ for any $g\in \A$ and $b\in\A$. If $U_{-2}<\beta(a_{-3})$, we set $X^{(-3)}_{-2}=k$ if and only if $U_{-2} \in I_1 (k, U_{-3})$, that is, $U_{-2}$ belongs to the $k$-th sub-interval of $[\beta, \beta( a_{-3}))$, which has size $\alpha(k|a_{-3})$. In the contrary case, we keep the temporary value $X^{(-3)}_{-2}=*$. In this example, $U_{-2}<\beta(a_{-3})$ and we denote $a_{-2}$ the index of
the interval to which $U_{-2}$ belongs (i.e, $U_{-2}\in I_1 (a_{-2}, U_{-3})$). 

As we have obtained $X^{(-3)}_{-3}$ and $X^{(-3)}_{-2}$, we now  check if this is enough to obtain the value of $X^{(-3)}_{-1}$. We consider a partition $ (I_2 (k, U_{-3}^{-2}), 1 \le k \le m )$ of the interval $[\beta, \beta( a_{-3}^{-2}))$ containing $m$ right-open intervals of size $\alpha(k|a_{-3}^{-2})-\alpha(k)$, respectively.
In this example, $U_{-1}>\beta(a_{-3}^{-2})$ and then we keep $X^{(-3)}_{-1}=*$.

Finally, let us check if we have enough information to obtain the value of $X^{(-3)}_{0}$. 
We consider a partition $ (I_3 (k, U_{-3}^{-1}), 1 \le k \le m )$ of the interval $[\beta, \beta(* a_{-3}^{-2}))$ containing $m$ right-open intervals of size $\alpha(k|*a_{-3}^{-2})-\alpha(k)$, respectively.
In this example, $U_{0}<\beta(*a_{-3}^{-2})$ and then we set $X^{(-3)}_{0}=a_0$, the index of the interval to which $U_{0}$ belongs (i.e, $U_{0}\in I_3 (a_{0}, U_{-3}^{-1})$).

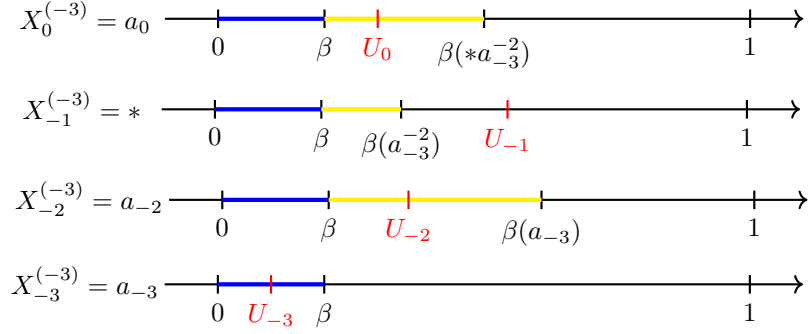
\begin{figure}[H]
\centering
\begin{tikzpicture}[scale=7]
\draw[->, thick] (-0.1,0) -- (1.1,0);
\foreach \x/\xtext in {0/0,0.2/$\beta$,0.5/ $\beta(*a_{-3}^{-2})$,1/$1$}
    \draw[thick] (\x,0.5pt) -- (\x,-0.5pt) node[below] {\xtext};
\draw[-, ultra thick, blue] (0,0) -- (0.2,0);
\draw[-, ultra thick, yellow] (0.2,0) -- (0.5,0);
\draw[thick, red] (0.3,0.5pt) -- (0.3,-0.5pt) node[below] {$U_0$};
\draw (-0.25,0) node {$X^{(-3)}_0=a_0$};
\end{tikzpicture}
\begin{tikzpicture}[scale=7]
\draw[->, thick] (-0.1,0) -- (1.1,0);
\foreach \x/\xtext in {0/0,0.2/$\beta$,0.35/$\beta( a^{-2}_{-3})$,1/$1$}
    \draw[thick] (\x,0.5pt) -- (\x,-0.5pt) node[below] {\xtext};
\draw[-, ultra thick, blue] (0,0) -- (0.2,0);
\draw[-, ultra thick, yellow] (0.2,0) -- (0.35,0);
\draw[thick, red] (0.55,0.5pt) -- (0.55,-0.5pt) node[below] {$U_{-1}$};
\draw (-0.25,0) node {$X^{(-3)}_{-1}=*$};
\end{tikzpicture}
\begin{tikzpicture}[scale=7]
\draw[->, thick] (-0.1,0) -- (1.1,0);
\foreach \x/\xtext in {0/0,0.2/$\beta$,0.6/$\beta(a_{-3})$,1/$1$}
    \draw[thick] (\x,0.5pt) -- (\x,-0.5pt) node[below] {\xtext};
\draw[-, ultra thick, blue] (0,0) -- (0.2,0);
\draw[-, ultra thick, yellow] (0.2,0) -- (0.6,0);
\draw[thick, red] (0.35,0.5pt) -- (0.35,-0.5pt) node[below] {$U_{-2}$};
\draw (-0.25,0) node {$X^{(-3)}_{-2}=a_{-2}$};
\end{tikzpicture}
\begin{tikzpicture}[scale=7]
\draw[->, thick] (-0.1,0) -- (1.1,0);
\foreach \x/\xtext in {0/0,0.2/$\beta$,1/$1$}
    \draw[thick] (\x,0.5pt) -- (\x,-0.5pt) node[below] {\xtext};
\draw[-, ultra thick, blue] (0,0) -- (0.2,0);
\draw[thick, red] (0.1,0.5pt) -- (0.1,-0.5pt) node[below] {$U_{-3}$};
\draw (-0.25,0) node {$X^{(-3)}_{-3}=a_{-3}$};
\end{tikzpicture}
\label{fig:5}
\end{figure}
Since we have finally obtained a non-star value for the process at time $0, $ we can stop the algorithm at this point.

\subsection{A perfect simulation algorithm with spontaneous symbols}

In the following $(U_n)_{n\in \Z}$ are i.i.d. uniform random variables in $[0,1)$.
Consider a countable alphabet $\A$, a set of admissible histories $\H \subset \A^{\N}$ associated to a probability kernel $p:\A\times \H \to [0,1]$. Assume that the kernel $p$ satisfies $\beta>0$.
The procedure  we propose to obtain a sample $X_{n-k},\ldots, X_n$ for some fixed $ k\geq0$ of a chain compatible with $p$ is given by Algorithm \ref{algo:ps}, and it is described below.

In what follows, for any $ n, $ we will introduce recursively a sequence of temporary values $X^{{(n-j)}}_{n}$, where the iteration is with respect to $j\geq 0, $ the number of steps of the algorithm. More precisely, $X^{{(n-j)}}_{n}$ is the temporary value of $X_n$ based only on $U_{n-j}^{n}$. The exponent $ (n-j)$ indicates up to which past time we have to go back in the past when considering the past values of the sequence $U_{-\infty}^n $ before time $n.$ Notice that this involves $j+1$ steps of the algorithm.

\begin{definition}  \label{def:algo1}
For any $n\in \Z$, we define
$$
X_n^{(n)}\mydef \sum_{g\in \A}g\mathbf{1}\left\{U_n \in \left[\sum_{b<g}\alpha(b),\sum_{b\leq g}\alpha(b)\right)\right\}+*\mathbf{1}\left\{U_n\geq \beta\right\}.
$$
With the convention $\inf\emptyset=*$, we have that
$$
X_{n}^{(n)} = \inf\left\{g\in \A:U_{n} < \sum_{b\leq g}\alpha(b)\right\}.
$$
With this convention, for any $m\in \Z$ and $n<m$, we define $X_{m}^{(n)}=X_{m}^{(n+1)}$ if $X_{m}^{(n+1)} \in \A$, and
$$
X_{m}^{(n)} := \inf\left\{g\in \A :U_{m} < \beta((X^{(n+1)})^{m-1}_{n+1})+\sum_{b\leq g}\alpha(b|(X^{(n)})^{m-1}_{n})-\alpha(b|(X^{(n+1)})^{m-1}_{n+1})\right\},
$$
in the contrary case. Notice that in the above recursion, $(X^{(n)})^{m-1}_{n}$ has already been attributed a value. Finally, we recall that by convention, $ \beta((X^{(n+1)})^{m-1}_{n+1}) = \beta $ in the case $ n= m-1.$ 

\end{definition}

Let us explain the above definitions. Suppose for simplicity that we only want to simulate a single value $X_n$ for a fixed time $n\in \Z$.
We first obtain $X_n^{(n)}$, which is the temporary value of $X_n$ based only on $U_n$.
If $X_{n}^{(n)}\in \A$, then we can stop and the algorithm returns $X_n=X_{n}^{(n)}$. If not, we obtain $X_{n-1}^{(n-1)}$, the temporary value of $X_{n-1}$ based only on $U_{n-1}$, and $X_n^{(n-1)}$, the temporary value of $X_{n}$ based only on $U_{n-1}$ and $U_{n}$. 
If $X_{n}^{(n-1)}\in \A$, the algorithm returns $X_n=X_{n}^{(n-1)}$. If not, we obtain $X_{n-2}^{(n-2)}$, $X_{n-1}^{(n-2)}$ and $X_{n}^{(n-2)}$ and so on. 

In Algorithm \ref{algo:ps} we keep the convention $\inf\emptyset=*$.

\begin{algorithm}[H]
\caption{PERFECT SIMULATION ALGORITHM TO SAMPLE $X_{-k}, \ldots,X_{0}$}
\label{algo:ps}
\begin{algorithmic}
\State Inputs:  $k \in \N$. Outputs: $X_{-j}$, for all $j={-k,\ldots,0}$.
\State Initialization: $X_{-j}^{(-n)}\leftarrow*$, for all $j={0,\ldots, k}$ and $n=j-1, j-2, \ldots, -1$.
\State $n\leftarrow -1$
\While{$X_{-j}^{(-n)} =*$, for any $j={0,\ldots,k}$}
\State $n\leftarrow n+1$
\State Draw the random variable $U_{-n}$.
\State $X_{-n}^{(-n)} \leftarrow \displaystyle \inf\left\{g\in \A:U_{-n} <\sum_{b\leq g}\alpha(b)\right\}$ 
\For{$m=n-1, n-2, \ldots, 0$}
\If{$X_{-m}^{(-n+1)} =*$ and $X_{-n}^{(-n)} \neq *$}
\State $
X_{-m}^{(-n)} \leftarrow$ \\ $\hspace{0.8cm}\displaystyle\inf\left\{g:U_{-m} < \beta((X^{(-n+1)})^{-m-1}_{-n+1})+\sum_{b\leq g}\alpha(b|(X^{(-n)})^{-m-1}_{-n})-\alpha(b|(X^{(-n+1)})^{-m-1}_{-n+1})\right\}  
$

\Else
\State $X_{-m}^{(-n)}=X_{-m}^{(-n+1)}$.
\EndIf
\EndFor
\EndWhile
\State \Return{$X_{-k}^{(-n)},\ldots,X_{0}^{(-n)}$}.
\end{algorithmic}
\end{algorithm}

\begin{definition}\label{def:T}
For any $n\in \Z$, let
\begin{equation}\label{eq:Tn}
T[n]:=\sup\{j\leq n: X_n^{(j)}\neq *\}.
\end{equation}
For any $m,n\in \Z$ such that $m<n$, let
$$
T[m,n]\mydef\inf\{T[k],k=m,\ldots,n\} ,
$$
and for any $ m \in \Z,$ let
$$
T[m, \infty)  \mydef\inf\{T[k],k \geq m \}.
$$  
\end{definition}

Note that $-T[n]$ (recall \eqref{eq:Tn}) is a stopping time with respect to the filtration $ ({\mathcal F}_k^{(n)})_{k \geq 0}, $ where for any $ k \geq 0,$ ${\mathcal F}_k^{(n)} = \sigma \{ U_{n-k}, \ldots, U_{n-1}, U_n \} . $ For any $n\in \Z$, on the event $\{|T[n]|<\infty\}$, the value
\begin{equation}
X_n=X_n^{(T[n])}
\end{equation}
is exactly the value that Algorithm \ref{algo:ps} returns. Note that $X_n$ does not depend on past values of the uniform random variables before time $ T[n].$ Finally, note that
$$
X_{m}^{(m-k)}=
\begin{cases}
*, &\text{ if } m-k>T[m],\\
X_m, &\text{ if } m-k\leq T[m].
\end{cases}
$$ 
In Subsection \ref{sec:results} we will give conditions to guarantee that $|T[n]|<\infty$, a.s.

\subsection{A perfect simulation algorithm without spontaneous symbols}

The assumption $\beta>0$ in Algorithm \ref{algo:ps} makes the spontaneous occurrence of symbols possible. Each obtained spontaneous symbol works as a source for updating the following symbols. Without assuming $\beta>0$, we need a different source in order to have a perfect simulation algorithm. In the following, to obtain a new source, we will follow the original idea proposed by \cite{cftp}. This means that we will simultaneously construct the process starting from all possible past strings of a certain length. When all these simultaneous processes hit a single symbol, this symbol will be the value of the original process that we want to simulate at that time. The symbols generated when the processes simultaneously hit a symbol are the source we need to replicate the procedure introduced by Algorithm \ref{algo:ps}. In the following, we propose an algorithm without the assumption $\beta>0$ and with the mechanism described above. Let us consider some initial definitions.

\begin{definition} \label{def:ptildeandpmarkov}
Given a probability kernel $p$, for any $n\geq 1$ we define 
$$
\beta_n  := \inf\{ \beta( a_{- n}^{-1}):  a_{-\infty}^{-1} \in \H\}.
$$
For any $n\geq 1$ satisfying $\beta_n >0 $ we define 
the transition matrix 
$$
p_{Markov}^{[n]}:\A\times\{a_{-n}^{-1}\in\A^n: a_{-\infty}^{-1}\in \H \}\to [0,1)
$$
by
$$
p_{Markov}^{[n]}(g|b_{-n}^{-1})=\frac{\alpha(g|b_{-n}^{-1})}{\beta(b_{-n}^{-1})}. $$
\end{definition}

The procedure  we propose to obtain a sample $X_{n-k},\ldots, X_n$ for some fixed $ k\geq0$ of a chain compatible with $p$ without the condition $\beta>0$ is given by Algorithm \ref{algo:ps2}. For this algorithm, we consider the following assumption. Let us introduce the notation
\begin{multline*}
 \hat n :=
\inf\{ n\in \N: \beta_n >0  \text{ and only one of the communicating classes associated to  } p^{[n]}_{Markov} \\
\text{ is closed, and this class is aperiodic}\},
\end{multline*}
where we put $\inf\emptyset=\infty$.

\begin{assumption} \label{as:algo2}
The alphabet $\A$ is finite. The probability kernel $p$ satisfies $\beta=0$ and  $\hat n<\infty$.

\end{assumption}

Let us denote $\mathcal{U}$ the unique closed class associated to $p^{[\hat{n}]}_{Markov}$. Recall that we intend to simultaneously construct the process
starting from all possible past strings of a certain length. This length will be $\hat{n}$. 
Let 
$$
\mathcal{C}:=\{a_{- \hat n}^{-1} \in \A^{\hat n}:  a_{-\infty }^{-1}\in \H \}
$$
and  $(U_n^{a_{- \hat n}^{-1}}:n\geq0,a_{- \hat n}^{-1}\in \mathcal{C})$ be i.i.d. uniform random variables in $[0,1)$. For any $n\geq 0$ and $a_{- \hat n}^{-1}\in \mathcal{C}$, define the process $(Y_n^{a_{-\hat n}^{-1}})_{n\geq 0}$ by
\begin{equation} \label{def:Ynprocess}
Y_n^{a_{-\hat{n}}^{-1}}=\inf\left\{g\in \A: U_n^{a_{- \hat n}^{-1}} <  \sum_{b\leq g}\alpha(b|(Y^{a_{-\hat{n}}^{-1}})_0^{n-1}a_{-\hat{n}}^{-1})\right\}, \mbox{ with the convention }\inf\emptyset=* .
\end{equation}
This process is a forward in time version of the update procedure proposed by Algorithm \ref{algo:ps} where we consider the fixed past $(Y^{a_{-\hat{n}}^{-1}})^{-1}_{-\hat{n}}=a_{-\hat{n}}^{-1}$. The collection of processes $(Y_n)_{n\geq 0}=((Y_n^{a_{-\hat n}^{-1}}:a_{-\hat n}^{-1}\in \mathcal{C}))_{n\geq 0}$ provides a coupled construction of the $|\mathcal{C}|$ processes starting from all possible initial pasts in $\mathcal{C}$. This coupled construction lets the different trajectories move independently.


For the next proposition, consider the following definition
$$
n_0:=\inf\left\{m\geq \hat{n}: \P\left((Y^{a_{-\hat{n}}^{-1}})^m_{m-\hat{n}+1}= b_{-\hat{n}}^{-1}, \text{ for all } a_{-\hat{n}}^{-1} \in \mathcal{C}\right)>0, \text{ for all }b_{-\hat{n}}^{-1}\in \mathcal{U} \right\}.
$$

\begin{proposition}
Suppose that Assumption \ref{as:algo2} holds. 
Then, $n_0<\infty.$
\end{proposition}
\begin{proof} 
By Assumption \ref{as:algo2}, we have that $\beta_{\hat n} >0$. Therefore, the probability that $(Y_n^{a_{-\hat n}^{-1}})_{n\geq 0}$ follows the $\hat n$-Markovian regime for all $a_{-\hat n}^{-1}\in \mathcal{C}$ and for all $n=0,1,\ldots, m$ is positive, for each fixed $m\geq 0$. 

By assumption, only one of the communication classes associated to the transition kernel $p^{[\hat n]}_{Markov}$ is closed, and this class is also aperiodic. This implies the following: If we consider independent copies of the Markov chain with transition matrix $p^{[\hat n]}_{Markov},$ starting from all possible values in $\mathcal{C}$, then there is a positive probability that after some time $m,$ all these Markov chains visit the same state $b_{-\hat n}^{-1}\in \mathcal{U}$ - a string of length $\hat n$ - at the same time (Chapter 11.4 of \cite{bremaud}). This property also holds for the original process following the kernel $p$, since it has a positive probability of following the $\hat n$-Markovian regime. This finishes the proof.  
\end{proof}

We now describe our algorithm. As before, let $(U_n^{a_{- \hat n}^{-1}}:n\geq0,a_{- \hat n}^{-1}\in \mathcal{C})$  be i.i.d. uniform random variables in $[0,1)$.
In the procedure proposed by Algorithm \ref{algo:ps2} we will use time windows $ I_j = \{ ( j-1) n_0 + 1 , \ldots,  j n_0 \} , j \in \Z. $  
\begin{enumerate}
    \item The algorithm considers sequentially the time windows of length $n_0$, starting from the window $I_0 .$ We consider the coupled construction defined in \eqref{def:Ynprocess} with all  possible initial states in $\mathcal{C}$. When these coupled chains hit simultaneously a single symbol which does not equal the star symbol, then the value of the process is the value that was hit simultaneously. 
    \item After considering the procedure above during the time window $I_{-k}, $ for a given $k\geq 1$, we have to update the value of the process for the times belonging to $ I_{-j}, $  $j=k-1,\ldots, 1$. For $n \in I_{-j}, $ the update for the value $X_{n}$ occurs only if the whole string $X_{-jn_0-\hat{n}}^{-jn_0}$ is known. These successive updates follow the ideas of the construction proposed in Algorithm \ref{algo:ps}. 
\end{enumerate}

Let us now describe the algorithm to simulate a single value of the process at a fixed time, say $X_0$. In this description, we will use the convention $\inf\emptyset=*$.
We first obtain $X_{m}^{(0)}$ for all $m\in I_0$, which are the temporary values of $X_{m}$, for $m\in I_0$. These values are constructed using the random variables $(U_{m}^{a_{-\hat{n}}^{-1}}: m \in I_0, a_{-\hat{n}}^{-1} \in \mathcal{C})$. To do this,  we consider for  $a_{-\hat{n}}^{-1} \in \mathcal{C}$ and $m \in I_0$,
$$
X^{a_{-\hat{n}}^{-1}}_{m} =\inf\left\{g\in \A:U_{m}^{a_{-\hat{n}}^{-1}} <\sum_{b\leq g}\alpha(b|(X^{ a_{-\hat{n}}^{-1}})_{-n_0+1}^{m-1}a_{-\hat{n}}^{-1})\right\}.
$$
Notice that this prescription corresponds to a backwards in time version of (\ref{def:Ynprocess}), in the spirit of the Propp and Wilson algorithm. 

For any $a_{-\hat{n}}^{-1} \in \mathcal{C}$, to decide about $X^{a_{-\hat{n}}^{-1}}_{m},$ for any $ m \in I_0, $ we consider the fixed past $X_{-n_0-\hat{n}+1}^{-n_0}=a_{-\hat{n}}^{-1}$ to update the following values using the procedure proposed by Algorithm \ref{algo:ps}. The values of $X^{a_{-\hat{n}}^{-1}}_{m}$ and $X^{b_{-\hat{n}}^{-1}}_{m},$ for $m\in I_0$ and $a_{-\hat{n}}^{-1} \neq b_{-\hat{n}}^{-1}$ are independent.  For $m \in I_0,$ 
if $X^{a_{-\hat{n}}^{-1}}_{m}=g$ for all $a_{-\hat{n}}^{-1} \in \mathcal{C}$ and for some $g\in \A$, we set $X_{m}^{(0)}=g$. 
In other words, we obtain the value of the chain when the trajectories starting from all possible past strings visit the same point. If this is not the case, we set $X_{m}^{(0)}=*$. Finally, 
if $X_{0}^{(0)}\in \A$, the algorithm returns $X_0=X_{0}^{(0)}$. This is the first round of the algorithm. 

If during the first round of the algorithm we were not able to find a value for $X_0,$ we go further back to the past and inspect the next time window $I_{-1}$ and obtain $X_{m}^{(-1)}, m \in I_{-1} ,$  the temporary values of $X_{m}, m \in I_{-1}. $ These values depend on $U_{m}^{a_{-\hat{n}}^{-1}}, m \in I_{-1},  a_{-\hat{n}}^{-1} \in \mathcal{C}$. To do this,  we consider for  $a_{-\hat{n}}^{-1} \in \mathcal{C}$ and $m\in I_{-1}$,
$$
X^{a_{-\hat{n}}^{-1}}_{m} =\inf\left\{g\in \A:U_{m}^{a_{-\hat{n}}^{-1}} <\sum_{b\leq g}\alpha(b|(X^{ a_{-\hat{n}}^{-1}})_{-2n_0+1}^{m-1}a_{-\hat{n}}^{-1})\right\}.
$$
Here, for any $a_{-\hat{n}}^{-1} \in \mathcal{C}$, to decide about $X^{a_{-\hat{n}}^{-1}}_{m}, $ for any $ m\in I_{-1}, $ we consider the fixed past $X_{-2n_0-\hat{n}+1}^{-2n_0}=a_{-\hat{n}}^{-1}$ to update the following values using the procedure proposed by Algorithm \ref{algo:ps}. For $m \in I_{-1}$, if $X^{a_{-\hat{n}}^{-1}}_{m}=g$ for all $a_{-\hat{n}}^{-1} \in \mathcal{C}$ and for some $g\in \A$, we set $X_{m}^{(-1)}=g$. If not, we set $X_{m}^{(-1)}=*$. The superscript of $ X_m^{(-1)} $ indicates that we had to look back into the past up to the time window $ I_{-1}.$ 

Now, it is time to obtain $X_{m}^{(-1)},m\in I_0$. We first check if $(X^{(-1)})_{-n_0-\hat{n}+1}^{-n_0}=b_{-\hat{n}}^{-1}$, for some $ b_{-\hat{n}}^{-1} \in \mathcal{C}$ (in particular, there is no $*$ symbol in this sequence). If not, we set $X_{m}^{(-1)}=X_{m}^{(0)}$ for all $m \in I_0$. In the positive case, for $m \in I_0, $ if $X_{m}^{(0)}\neq *$ we set $X_{m}^{(-1)}=X_{m}^{(0)}$ and if $X_{m}^{(0)}= *$, we set $X_{m}^{(-1)}$ equal to 
$$
\inf\left\{g:U_{m}^{b_{-\hat{n}}^{-1}} < \beta((X^{b_{-\hat{n}}^{-1}})^{m-1}_{-n_0+1}b_{-\hat{n}}^{-1})+\sum_{b\leq g}\alpha(b|(X^{(-1)})^{m-1}_{-2n_0+1})-\alpha(b|(X^{b_{-\hat{n}}^{-1}})^{m-1}_{-n_0+1}b_{-\hat{n}}^{-1})\right\}.
$$
This means that we update the values of the process associated to the past string of length $\hat{n}$ that we have obtained, which is $(X^{(-1)})_{-n_0-\hat{n}+1}^{-n_0}=b_{-\hat{n}}^{-1}$. Note that in this update we use the uniform random variable $U_{m}^{b_{-\hat{n}}^{-1}}$, which is associated to the past we obtained. Finally, 
if $X_{0}^{(-1)}\in \A$, the algorithm returns $X_0=X_{0}^{(-1)}$, and so on. 

To have a complete description of the algorithm, we need to explain how we proceed in case we need to go one step further in the past. The construction of $X_m^{(-2)}, m\in I_{-2} \cup I_{-1}$ is analogous to the construction described above, the difference appears in the construction of $X_{m}^{(-2)},m\in I_0$, as we explain in the following. 
We first check if $(X^{(-2)})_{-n_0-\hat{n}+1}^{-n_0}=b_{-\hat{n}}^{-1}$, for some $ b_{-\hat{n}}^{-1} \in \mathcal{C}$.  If not, we set $X_{m}^{(-2)}=X_{m}^{(-1)}$ for all $m \in I_0$. In the positive case, for $m \in I_0, $ if $X_{m}^{(-1)}\neq *$ we set $X_{m}^{(-2)}=X_{m}^{(-1)}$. If $X_{m}^{(-1)}= *$, we now have two possibilities. If $(X^{(-1)})_{-n_0-\hat{n}+1}^{-n_0}\neq b_{-\hat{n}}^{-1}$ (i.e., we are in the iteration in which the symbols of the process for all times $-n_0-\hat{n}+1, \ldots, n_0$ are obtained for the first time),
we set $X_{m}^{(-2)}$ equal to 
$$
\inf\left\{g:U_{m}^{b_{-\hat{n}}^{-1}} < \beta((X^{b_{-\hat{n}}^{-1}})^{m-1}_{-n_0+1}b_{-\hat{n}}^{-1})+\sum_{b\leq g}\alpha(b|(X^{(-2)})^{m-1}_{-3n_0+1})-\alpha(b|(X^{b_{-\hat{n}}^{-1}})^{m-1}_{-n_0+1}b_{-\hat{n}}^{-1})\right\}.
$$
In the contrary case,  we have that $X_{-m}^{(-2)}$ is equal to
$$
\inf\left\{g:U_{-m}^{b_{-\hat{n}}^{-1}} < \beta((X^{(-1)})^{-m-1}_{-2n_0+1})+\sum_{b\leq g}\alpha(b|(X^{(-2)})^{-m-1}_{-3n_0+1})-\alpha(b|(X^{(-1)})^{-m-1}_{-2n_0+1})\right\}.
$$

Let $l_j=n_0(j-1)+1$, and $r_j=n_0j$  denote the smallest and the greatest value of the set $I_j$,  respectively, for $j\in \Z$. The above description can be generalized to describe the construction of $X_{r_n}^{(m)}, X_{r_n-1}^{(m)} \ldots, X_{l_n}^{(m)}$, for any $n\in \Z$ and for any $m\leq n$. This means that the construction we are considering always starts from times of the form $r_n$, for $n\in \Z$.  In Algorithm \ref{algo:ps2} we keep the convention $\inf\emptyset=*$ 

\begin{algorithm}[ht!]
\caption{PERFECT SIMULATION ALGORITHM TO SAMPLE $X_{-k}, \ldots,X_0$}
\label{algo:ps2}
\begin{algorithmic}
\State Inputs:  $k \in \N$. Outputs: $X_{-j}$, for all $j={-k,\ldots,0}$.
\State Initialization: $X_{-j}^{(-n)}\leftarrow*$, for all $j={0,\ldots, k}$ and $n=j-1, j-2, \ldots, -1$.
\State $n\leftarrow -1$
\While{$X_{-j}^{(-n)} =*$, for any $j={0,\ldots,k}$}
\State $n\leftarrow n+1$
\For{$m=l_{-n}, l_{-n}+1, \ldots, r_{-n}$:}
\For{$a_{-\hat{n}}^{-1} \in \mathcal{C}$}
\State Draw the random variable $U_{m}^{a_{-\hat{n}}^{-1}}$.
\State $X^{ a_{-\hat{n}}^{-1}}_{m} \leftarrow \displaystyle \inf\left\{g\in \A:U_{m}^{a_{-\hat{n}}^{-1}} <\sum_{b\leq g}\alpha(b|(X^{a_{-\hat{n}}^{-1}})_{l_{-n}}^{m-1}a_{-\hat{n}}^{-1})\right\}$
\EndFor
\If{$X^{a_{-\hat{n}}^{-1}}_{m}=g$, for all $a_{-\hat{n}}^{-1} \in \mathcal{C}$ and for some $g\in \A$}
\State $X_{m}^{(-n)}=g$.
\Else
\State $X_{m}^{(-n)}=*$.
\EndIf
\EndFor
\For{$k=n-1,n-2,\ldots, 0$:}
\If{$(X^{(-n)})_{l_{-k}-\hat{n}}^{l_{-k}-1}=b_{-\hat{n}}^{-1}$, for some $b_{-\hat{n}}^{-1} \in \mathcal{C}$}
\For{$m=l_{-k}, \ldots, r_{-k}$:}
\If{$X_{m}^{(-n+1)} = *$}
\If{$(X^{(-n+1)})_{l_{-k}-\hat{n}}^{l_{-k}-1}=b_{-\hat{n}}^{-1}$}
\State $
 X_{m}^{(-n)} \leftarrow$ \\ \\ $\hspace{0.5cm}\displaystyle\inf\left\{g:U_{m}^{b_{-\hat{n}}^{-1}} < \beta((X^{(-n+1)})^{-m-1}_{l_{-(n-1)}})+\sum_{b\leq g}\alpha(b|(X^{(-n)})^{m-1}_{l_{-n}})-\alpha(b|(X^{(-n+1)})^{m-1}_{l_{-(n-1)}})\right\} $
\Else
\State $
X_{m}^{(-n)} \leftarrow$ \\ \\ $\hspace{0.5cm}\displaystyle\inf\left\{g:U_{m}^{b_{-\hat{n}}^{-1}} < \beta((X^{b_{-\hat{n}}^{-1}})^{m-1}_{l_{-k}}b_{-\hat{n}}^{-1})+\sum_{b\leq g}\alpha(b|(X^{(-n)})^{m-1}_{l_{-n}})-\alpha(b|(X^{b_{-\hat{n}}^{-1}})^{m-1}_{l_{-k}}b_{-\hat{n}}^{-1})\right\} $
\EndIf
\Else
\State $X_{m}^{(-n)} \leftarrow X_{m}^{(-n+1)}$.
\EndIf
\EndFor
\Else
\State $X_{m}^{(-n)} \leftarrow X_{m}^{(-n+1)}$.
\EndIf
\EndFor
\EndWhile
\Return{$X_{-k}^{(-n)},\ldots,X_{0}^{(-n)}$}.
\end{algorithmic}
\end{algorithm}

\begin{remark}
We considered a coupling construction in which all the trajectories evolve independently. Different constructions would be possible, the choice we made is the one which simplifies the presentation and readability of the algorithm. 
\end{remark}

\begin{definition}\label{def:Ttilde}
For any $z \in \Z$ and for any $n\in I_z$, let
\begin{equation}\label{eq:Tntilde}
\tilde{T}[n]:=\sup\{j\leq z: X_n^{(j)}\neq *\}.
\end{equation}
We define for any $n,m\in \Z$, with $m<n$,
$$
\tilde{T}[m,n]\mydef\inf\{\tilde{T}[k],k=m,\ldots,n\}
$$
and for any $ m \in \Z,$
$$
\tilde{T}[m, \infty)  \mydef\inf\{\tilde{T}[k],k \geq m \}.
$$
\end{definition}

Note that $\tilde{T}[n]=k$ means that $X_n$ is obtained from $U_{l_k},U_{l_k+1}, \ldots, U_n$.
For any $n\in \{0,1,\ldots\}$, on the event $\{|\tilde{T}[n]|<\infty\}$, the value
\begin{equation}
X_n=X_n^{(\tilde{T}[n])}
\end{equation}
is exactly the value that Algorithm \ref{algo:ps2} returns. Note that $X_n$ does not depend on past values of the uniform random variables before time $l_{T[n]}$. 
Finally, note that
$$
X_{m}^{(m-k)}=
\begin{cases}
*, &\text{ if } m-k>\tilde{T}[m],\\
X_m, &\text{ if } m-k\leq \tilde{T}[m].
\end{cases}
$$ 
In Subsection \ref{sec:results} we will give conditions to guarantee that $|\tilde{T}[n]|<\infty$, a.s.

\subsection{Main results} \label{sec:results}

We first give two Theorems stating that Algorithms \ref{algo:ps} and \ref{algo:ps2} belong to the family of algorithms called {\it coupling from the past} in the literature. 

\begin{theorem} \label{teo:cftp}
Suppose that $|T[n]|<\infty$ a.s. for all $n\in \Z$.
Let $(X_n)_{n\in \Z}$ be the output of Algorithm \ref{algo:ps}. Then there exists a function
$F:[0,1)^{\N} \to \A$,
such that for any $n\in \Z$,
$$
X_n=F(U_n,U_{n-1}, \ldots, U_{T[n]}, u_{-1},u_{-2},\ldots)
$$
a.s., for any $u_{-1},u_{-2},\ldots \in [0,1)$. 
\end{theorem}

The proof of the above result is given in Section \ref{sec:proofteocftp} where we construct a function $F$, called {\it coupling function} in the sequel, with the desired properties.

\begin{theorem} \label{coro:cftp}
Suppose that Assumption \ref{as:algo2} holds and suppose that $|\tilde{T}[n]|<\infty$ a.s. for all $n\in \Z$.
Let $(X_n)_{n\in \Z}$ be the output of Algorithm \ref{algo:ps2}. Then there exist functions
$$
\tilde{F}^{(j)}: [0,1)^{\N} \to \A, \text{ for } j=0,\ldots, n_0-1,
$$
such that for any $n\in \Z$ and for any $j=1,\ldots,n_0 $
$$
X_{nn_0-j}=\tilde{F}^{(j)}(U_{nn_0-j},U_{nn_0-j-1}, \ldots, U_{l_{\tilde{T}[nn_0-j]}}, u_{-1},u_{-2},\ldots) 
$$
a.s., for any $u_{-1},u_{-2},\ldots \in [0,1)$. 
\end{theorem}

We now state two results, Theorems \ref{teo:infinitepath} and \ref{teo:renewal}, which provide sufficient conditions guaranteeing that the assumptions of Theorem \ref{teo:cftp} hold for Algorithm \ref{algo:ps}.

For any $n\geq 1$, let
\begin{equation} \label{def:rho}
\rho_n:=\sum_{a_{-n}^{-1}\in \A^n}\alpha(a_{-n})\alpha(a_{-(n-1)}|a_{-n})\ldots \alpha(a_{-1}|a_{-n}^{-2}).
\end{equation}
Since
\begin{multline*}
\rho_n=\sum_{a_{-n}^{-2}\in \A^{n-1}}\alpha(a_{-n})\alpha(a_{-(n-1)}|a_{-n})\ldots \alpha(a_{-2}|a_{-n}^{-3})\sum_{a_{-1}\in \A}\alpha(a_{-1}|a_{-n}^{-2})\leq \\ \sum_{a_{-n}^{-2}\in \A^{n-1}}\alpha(a_{-n})\alpha(a_{-(n-1)}|a_{-n})\ldots \alpha(a_{-2}|a_{-n}^{-3}) =\rho_{n-1},
\end{multline*}
 $(\rho_n)_{n\geq 1}$ is a non-increasing sequence of positive numbers, which therefore admits a limit.

\begin{theorem} \label{teo:infinitepath}
Suppose that there exists $c>0$ such that
$$
\lim_{n\to \infty}\rho_n\geq c.
$$
Then, the following holds.
\begin{enumerate}
\item $\E(|T[0]|)\leq \frac{1-c}{c}.$

\item For any $n\geq 1$, $\E(|T[0,n]|)\leq \frac{(1-c)(n+1)}{c}$.

\item $\P(|T[0,+\infty)|<+\infty) =1$.
\end{enumerate}
\end{theorem}

\begin{theorem} \label{teo:renewal}
Suppose that
$$
\sum_{n=1}^{\infty}\rho_n=\infty.
$$
Then, for any $n\geq 1$, $\P(|T[0,n]|<+\infty)=1$.
\end{theorem}

We now state two results, Theorems \ref{teo:algops2_1} and \ref{teo:algops2_2}, which provide sufficient conditions guaranteeing that the assumptions of Theorem \ref{coro:cftp} hold for Algorithm \ref{algo:ps2}.

Recall that  $\mathcal{U}$ denotes the unique closed class associated to $p^{[\tilde{n}]}_{Markov}.$ For any $n\geq 1$, let
$$
\tilde{\rho}_n:=\sum_{x_{-\hat{n}}^{-1}\in \mathcal{U}}\sum_{a_{-n}^{-1}\in \A^n}\alpha(a_{-n}|x_{-\hat{n}}^{-1})\alpha(a_{-(n-1)}|a_{-n}x_{-\hat{n}}^{-1})\ldots \alpha(a_{-1}|a_{-n}^{-2}x_{-\hat{n}}^{-1}).
$$
Note that $(\tilde{\rho}_n)_{n\geq 1}$ is a non-negative and non-increasing sequence, and therefore admits a limit.

\begin{theorem}\label{teo:algops2_1}
Suppose that there exists $c>0$ such that
$$
\lim_{n\to \infty}\tilde{\rho}_n\geq c,
$$
Then, the following holds.
\begin{enumerate} 
\item $\E(|\tilde{T}[0]|)<\infty.$

\item For any $n\geq 1$, $\E(|\tilde{T}[0,n]|)<\infty$.

\item $\P(|\tilde{T}[0,+\infty)|<+\infty) =1$.
\end{enumerate}
\end{theorem} 

\begin{theorem}\label{teo:algops2_2}
Suppose that
$$
\sum_{n=1}^{\infty}\tilde{\rho}_n=\infty
$$
Then, for any $n\geq 1$, $\P(|\tilde{T}[0,n]|<+\infty)=1$.
\end{theorem}

The coupling functions considered in Algorithms \ref{algo:ps} and \ref{algo:ps2} are original, since they allow for unknown positions -- those taking the symbol $*$ -- in the past.  
In the coupling from the past framework, other articles consider perfect simulation algorithms based on coupling functions for chains with infinite memory. \cite{cff} and \cite{desantispiccioni}
consider a perfect simulation algorithm based on the construction of a function
$
G:[0,1]\times \H \to \A
$
for which a similar result as in Theorem \ref{teo:cftp} or Theorem \ref{coro:cftp} holds. Once such a construction is achieved, under suitable assumptions, the sequence $(\tilde{X}_n)_{n\in \Z}$ is easily shown to be the unique (in law) process which is compatible with the kernel $p$. Similar ideas have been presented in \cite{gallo}, where a perfect simulation algorithm for chains with memory of variable length is introduced, 
in \cite{gallogarcia}, where chains of infinite order are considered having a transition kernel that is only locally continuous, and in \cite{gallotakahashi}, where regular probability kernels $p$ on a finite alphabet $\A$ are studied. 

This being said, once the construction of a perfect simulation algorithm is achieved, the same consequences as those presented in the above cited works hold also in our case.  In particular we have existence and uniqueness of a stationary measure (Corollaries 4.1 and 
Proposition 6.1 (iii) of \cite{cff} and Corollary 4.1 of \cite{gallogarcia}), the loss of memory property (Corollary 4.1 of \cite{cff}), the existence of a regeneration scheme (Corollary 6.1 of \cite{cff} and Corollary 2 of \cite{gallo}) and the concentration of measure phenomenon (Corollary 1 of \cite{gallotakahashi}).
These results are summarized in the next propositions. Their proofs, which are omitted, use standard arguments and can be found in the articles cited above.

The following propositions are stated considering $(X_n)_{n\in \Z}$ as the output of Algorithm \ref{algo:ps} and the associated times as in Definition \ref{def:T}. These propositions also hold when we consider $(X_n)_{n\in \Z}$ as the output of Algorithm \ref{algo:ps2} and the associated times as in Definition \ref{def:Ttilde}.

\begin{proposition} \label{prop:consequences}
Suppose that $\P(|T[m,n]|<\infty)=1$, for any $m,n \in \Z$, $m<n$. Then the following holds.
\begin{enumerate}
\item $(X_n)_{n\in \Z}$ is the unique stationary chain compatible with $p$.  
\item (Loss of memory) For any $k\leq 1$, $a^{-1}_{-\infty}\in \A^{\N}$ and $b^{-1}_{-\infty}\in \A^{\N}$,
$$
|\P(X_0 \in \cdot |X_{-\infty}^{-k}=a^{-k}_{-\infty}) -\P(X_0 \in \cdot |X_{-\infty}^{-k}=b^{-k}_{-\infty})| \leq \P(|T[0]|\geq  k).
$$
\end{enumerate}
    
\end{proposition}

\begin{proposition} \label{prop:consequences2}
Suppose that $\P(|T[0,+\infty)|<\infty)=1$. Then the sequence $(\mathbf{1}\{T[n,+\infty)=n\})_{n\in \Z}$ is a stationary renewal process. This implies that $(X_n)_{n\in \Z}$ has a regeneration scheme.
\end{proposition}

The following concentration of measure result is due to Corollary $1$ of \cite{gallotakahashi}. Let $f:\A^n \to \R$ be measurable. Define $\delta f=(\delta_1 f, \ldots, \delta_n f)$, where for any $j=1,\ldots, n$,
$$
\delta_jf=\sup\{|f(a_{1}^n)-f(b_{1}^n)|:a_i=b_i \text{ for } i\neq j\}.
$$

\begin{proposition} \label{prop:gallotakahashi}
Suppose that the alphabet $\A$ is finite and $\E(|T[0]|)<\infty$. Then the following holds.
For all $\epsilon>0$ and all functions $f:\A^n \to \R$,
$$
\P(|f(X_1^n)-\E[f(X_1^n)]|>\epsilon)\leq 4 \exp\left\{\frac{-2\epsilon^2}{9(1+\E(T[0]))^2||\delta f||^2_{\ell_2(\mathbb{N})}}.
\right\}
$$
\end{proposition}

Section 8 of \cite{gallogarcia} describes more consequences of the existence of a perfect simulation algorithm such as error bounds for the coupling of chains with infinite memory and Markov chains (see \cite{gallotakahashi2}), the control of the decay of correlations and functional central limit theorems. 

\section{Examples} \label{sec:examples}

Before giving the proofs of our main results, we present some examples. Examples \ref{ex:autoregressive},  \ref{ex:imitation} and \ref{ex:other} and Remark \ref{rem:imitationexample} deal with probability kernels for which the assumptions of Algorithm \ref{algo:ps} hold. Example \ref{ex:algo2}  and Remark \ref{remark:examplealgo2} deal with probability kernels for which the assumptions of Algorithm \ref{algo:ps2} hold. 
In Section \ref{sec:dicussion} we will come back to these examples to compare our results with previously obtained results in the literature. In the following, for any set $A$, $\mathbf{1}_A(\cdot)$ denotes the indicator function of $A$.

\begin{example} \label{ex:autoregressive}
Consider a collection of non-negative numbers $\{\theta_j:j\geq 0\}$  such that $\sum_{j=0}^{\infty}\theta_j=1$. Let us consider the linear autoregressive binary model (see \cite{autoregressive}) taking values in $\A=\{0,+1\}$ such that for any $x_{-\infty}^{-1}\in \H=\A^{\N}$, the probability kernel is  given by
$$
p(1|w_{-\infty}^{-1})= \theta_0(1-\delta) +\sum_{j=1}^{\infty}\theta_jw_{-j}.
$$
Here, $\delta \in (0,1)$. In this example, the state of the process at a time is either chosen independently of the past with probability $\theta_0$ or it is the copy of the symbol $k$ steps in the past with probability $\theta_k$.

Assume $\theta_0>0$. For this kernel, we have $\beta=\theta_0$ and for any $w_{-n}^{-1} \in \A^n, n\geq 1$, we have $\beta(w_{-n}^{-1})=\sum_{j=0}^{n}\theta_j$. This implies that for $n\geq 1$, by denoting $s_n=\theta_{n}+\theta_{n+1}+\ldots$ we have that
$$
\rho_n= \prod_{j=1}^{n}(1-s_j).
$$
Note that this is the probability of first choosing a spontaneous symbol, then either choosing a spontaneous symbol or copying the symbol one step in the past, then  either choosing a spontaneous symbol or copying a symbol at most two steps back in the past, and so on.  
If we suppose that
\begin{equation}\label{eq:autoregressive}
    \lim_{n\to \infty} \prod_{j=1}^{n}(1-s_j)>0,
\end{equation}
then Theorem \ref{teo:infinitepath} holds.
The condition above is equivalent to
$
\sum_{m=1}^{\infty} \ln (1-s_m)> -\infty.
$
Since $s_n\to 0$ as $n\to \infty$, this is equivalent to
$
\sum_{m=1}^{\infty}s_m< \infty.
$
Therefore, we have to assume that the mean memory length  is finite. Here, by {\it memory length} we mean the number of steps we look back in the past to copy the symbol at that past time.  

A condition that is sufficient to ensure that Theorem \ref{teo:renewal} holds is the following
$$
\sum_{n=1}^{\infty}\prod_{j=1}^{n}(1-s_j)=+\infty.
$$
By noting that
$$
\frac{\prod_{j=1}^{n+1}(1-s_j)}{\prod_{j=1}^{n}(1-s_j)}=(1-s_{n+1}),
$$
by Raabe's test for infinite series (see \cite{infiniteseries}), we have that Theorem \ref{teo:renewal} holds if there exists $n_0\in \{1,2,\ldots\}$ and $\epsilon >0$ such that
$$
s_n\leq \frac{(1-\epsilon)}{n},
$$
for all $n\geq n_0$. Therefore, Theorem \ref{teo:renewal} holds even in some cases in which the mean memory length is not finite.
\end{example}

The following two examples are imitation models (see \cite{cff,imitation_noise}) with memory of variable length.

\begin{example} \label{ex:imitation}

Consider $\A=\{1,2,\ldots\}$ and let $ \{c_g, g \in \A \}$  be a collection of nonnegative numbers such that $c=\sum_{g\in \A}c_g <1.$ For any $x_{-\infty}^{-1}\in \H=\A^{\N}$ and for any $g\in \A$, let
\begin{equation}\label{eq:bramson-kalikow}
p(g|x_{-\infty}^{-1})=c_g+(1-c)\frac{\displaystyle \sum_{k=1}^{x_{-1}}\mathbf{1}\{x_{-k}=g\}}{x_{-1}}.
\end{equation}
In this example, the state of the process at a time is either chosen independently of the past with probability $c$ or it depends on the past symbols with probability $1-c$. In the first case, symbol $g\in \A$ is chosen with probability $c_g/c$. In the second case, we choose uniformly and copy one of the last $x_{-1}$ symbols of the past, where $x_{-1} \in \A$ is the first symbol in the past of the process.

Clearly, 
\begin{multline*}
\sum_{a_{-n}^{-1}\in \A^n}\alpha(a_{-n})\alpha(a_{-(n-1)}|a_{-n})\ldots \alpha(a_{-1}|a_{-n}^{-2})\geq \\ 
\sum_{a_{-n}^{-1}:a_{-n+k-1}\leq k, k=1,\ldots, n }\alpha(a_{-n})\alpha(a_{-(n-1)}|a_{-n})\ldots \alpha(a_{-1}|a_{-n}^{-2}) \geq \\ 
c_1(c_1+c_2+(1-c))\ldots (c_1+\ldots +c_n+(1-c)).
\end{multline*}
Indeed, the right-hand side of the above inequality gives the probability of first choosing a spontaneous symbol which equals $ 1.$ Thus, if in the next step we either sample a spontaneous symbol that equals either $ 1 $ or $ 2 $ or if we decide, with probability $ 1 -c , $ to look into the past, then the memory length, which is given by the first symbol, is given by $1$. This argument can then be iterated. 
Note that the above event is a sort of ladder event $ X_1 = 1,$ followed by $ X_2 \le 2,$ $ X_3 \le 3 $ ($X_3 = 1 $ is a possible choice) and so on.

If we suppose that
\begin{equation}\label{eq:houseofcards}
    \lim_{n\to \infty} c_1(c_1+c_2+(1-c))\ldots (c_1+\ldots +c_n+(1-c))>0,
\end{equation}
that is, infinite ladder events have a positive probability, then Theorem \ref{teo:infinitepath} holds.

By denoting $s_m=c_m+c_{m+1}+\ldots,$ and proceeding as in Example \ref{ex:autoregressive}, the condition above is equivalent to
$$
\sum_{m=1}^{\infty} s_m< \infty.
$$
Therefore, we have to assume that the mean value of the spontaneously generated symbol (which also determines the memory length in this example) is finite.

As in Example \ref{ex:autoregressive}, a condition that is sufficient to ensure that Theorem \ref{teo:renewal} holds is 
that there exists $n_0\in \{1,2,\ldots\}$ and $\epsilon >0$ such that
$$
s_n\leq \frac{(1-\epsilon)}{n},
$$
for all $n\geq n_0$. This condition does not exclude the possibility of having
$$
\sum_{n\geq 1}s_n=+\infty,
$$
and therefore Theorem \ref{teo:renewal} holds even in some cases in which the mean value of the spontaneously generated symbol (and thus of the associated memory length) is not finite.

Note that in this example, we do not have a regeneration time which is a stopping time. By going forward in time, it is always possible to choose a memory length that goes back as many steps in the past as we want.

\end{example}

\begin{remark} \label{rem:imitationexample}
In Example \ref{ex:imitation}, the process uses only the last step to define the size of the memory. Moreover, it decides to copy uniformly one of the past symbols in the range of the memory length. This kernel can be easily generalized so that the presented results hold in a similar way.

Consider as a generalization the kernel
\begin{equation}\label{eq:bramson-kalikow3}
p(g|x_{-\infty}^{-1})=c_g+(1-c)\sum_{k=1}^{\infty}f_{x_{-1}}(k)\mathbf{1}\{x_{-k}=g\},
\end{equation}
for each $x^{-1}_{-\infty}\in \A^{\N}$. Here, $\sum_{k\geq 1}f_{x_{-1}}(k)=1$. 
For $m\geq 1$, we have  $\sum_{k=1}^{m}f_{m}(k)\geq d_m \in (0,1)$, with $d_m\to 1$ as $m\to \infty$.

This kernel generalizes the one in Example \ref{ex:imitation} in the following way. The position of the symbol which will be copied is chosen with a distribution which is not uniform. This distribution has support in $\{1,2,3,\ldots\}$ but it is concentrated in the most $x_{-1}$ recent symbols with high probability, where $x_{-1}$ is the first symbol in the past of the
process.

Now, we have
$$
\sum_{a_{-n}^{-1}\in \A^n}\alpha(a_{-n})\alpha(a_{-(n-1)}|a_{-n})\ldots \alpha(a_{-1}|a_{-n}^{-2})\geq c_1(c_1+c_2+(1-c)d_1)\ldots (c_1+\ldots c_n+(1-c)d_n) .
$$
As we did in Example \ref{ex:imitation}, we can consider $c_n \to 0$ and $d_n\to 1$ fast enough as $n\to \infty$ such that Theorems \ref{teo:infinitepath} and \ref{teo:renewal} hold.

\end{remark}

\begin{example} \label{ex:other}
Consider $\A=\{1,2,\ldots\}$ and let $ \{c_g, g \in \A \}$  be a collection of non negative numbers such that $c=\sum_{g\in \A}c_g <1.$ For each $k\geq 1$, let $(q_k(g):g\in \A)$ be a probability distribution. For any $x_{-\infty}^{-1}\in \H=\A^{\N}$ and for any $g\in \A$, let
\begin{equation}\label{eq:sec_example}
p(g|x_{-\infty}^{-1})=c_g+(1-c)q_{Y(x_{-\infty}^{-1})}(g),
\end{equation}
where 
$$
Y(x_{-\infty}^{-1})=\inf\left\{m\geq 1: \sum_{n=1}^mx_{-n} \leq \frac{m(m+1)}{2}\right\}.
$$
Recall that $\sum_{n=1}^mn=\frac{m(m+1)}{2}$.

In this example, the state of the process at a time is either chosen independently of the past with probability $c$ or it depends on the past symbols with probability $1-c$. In the first case, symbol $g\in \A$ is chosen with probability $c_g/c$. In the second case, we choose the new state using a probability distribution which depends only on $Y(x_{-\infty}^{-1})$. Note that for any $k\geq 1$ and for any $x_{-\infty}^{-1}$, to check if $Y(x_{-\infty}^{-1})=k$ we need to check only $x_{-k}^{-1}$, i.e., the first $k$ elements of $x_{-\infty}^{-1}$.

Let us suppose that $c_1>0$ and that for each $k\geq 1$, 
$$
\min\left\{\sum_{n=1}^{k+1}q_m(n):m\leq k\right\}\geq d_k,
$$
with $d_k \to 1$ as $k\to \infty$. This implies that for any $k\geq 1$ and $x_{-\infty}^{-1}$ satisfying $\sum_{n=1}^kx_{-n} \leq \frac{k(k+1)}{2}$, we have that
$$
\sum_{g\leq k+1}p(g|x_{-\infty}^{-1}) \geq c_1+\ldots +c_{k+1}+(1-c)\min\left\{\sum_{n=1}^{k+1}q_m(n):m\leq k\right\} \geq c_1+\ldots + c_{k+1}+(1-c)d_k.
$$
As a consequence, 
\begin{multline*}
\sum_{a_{-n}^{-1}\in \A^n}\alpha(a_{-n})\alpha(a_{-(n-1)}|a_{-n})\ldots \alpha(a_{-1}|a_{-n}^{-2})\geq \\ 
\sum_{a_{-n}^{-1}:a_{-n+k-1}\leq k, k=1,\ldots, n }\alpha(a_{-n})\alpha(a_{-(n-1)}|a_{-n})\ldots \alpha(a_{-1}|a_{-n}^{-2})\geq \\ 
 c_1(c_1+c_2+(1-c)d_1)\ldots (c_1+\ldots + c_n+(1-c)d_{n-1}).
\end{multline*}
Once more, we can find $c_n \to 0$ and $d_n\to 1$ fast enough as $n\to \infty$ such that the conditions of Theorems \ref{teo:infinitepath} and \ref{teo:renewal}  hold.

As in the preceding example, the first symbol in the past is necessary to obtain the memory length. If the actual symbol is not spontaneous, nor is the past symbol, there is no possibility to decide about the actual symbol by only looking to the spontaneous symbols of the past. 

For the same reasons as in Example \ref{ex:imitation}, we do not have a regeneration time which is a stopping time. 
    
\end{example}


\begin{example} \label{ex:algo2}
Consider a collection of non-negative numbers $\{\theta_j:j\geq 0\}$  such that $\sum_{j=0}^{\infty}\theta_j=1$. Suppose that $\theta_0>0$. Consider a model taking values in $\A=\{0,1,2,3\}$ with the convention $3+1=0$ and $0-1=3$. Let 
$$
\H=\{a_{-\infty}^{-1}\in \A^{\N}:a_{-n+1}\in \{a_{-n}-1,a_{-n},a_{-n}+1\}, \text{ for all } n \geq 2\}.
$$
For any $w_{-\infty}^{-1}\in \H$, the probability kernel is given as follows. For any $w_{-\infty}^{-1}\in \H$ and for $g=w_{-1}$,
$$
p(g|w_{-\infty}^{-1})= \frac{1}{3}\theta_0 +\sum_{j=1}^{\infty}\theta_j
\mathbf{1}_{\{w_{-j}, w_{-j}+2 \}}(g).
$$
For $g\in \{w_{-1}-1,w_{-1}+1\}$
$$
p(g|w_{-\infty}^{-1})= \frac{1}{3}\theta_0
+\sum_{j=1}^{\infty}\theta_j\mathbf{1}_{\{w_{-j}\}}(g).
$$
Finally, we define $p(g|w_{-\infty}^{-1})=0$ for $g=w_{-1}+2$.
By construction, the process can only stay on the same symbol or jump to one of the neighbors symbols. 

Observe that for $g,b\in \A$ such that $b\neq g+2$,
$$
\alpha(b|g)=\frac{\theta_0}{3}\mathbf{1}_{\{g-1,g,g+1\}}(b)+\theta_1\mathbf{1}_{\{g\}}(b).
$$
Recalling Definition \ref{def:ptildeandpmarkov}, we have that
$$
\beta_1=\theta_0+\theta_1
$$
and for any $g,b\in \A$ such that $b\neq g+2$,
$$
p_{Markov}^{[1]}(b|g)=\frac{1}{3}\frac{\theta_0}{\theta_0+\theta_1}\mathbf{1}_{\{g-1,g,g+1\}}(b)+\frac{\theta_1}{\theta_0+\theta_1}\mathbf{1}_{\{g\}}(b).
$$
This implies that Assumption \ref{as:algo2} holds with $\hat n=1$. Finally, note that $n_0=2$.

For this kernel, we have $\beta=0,$ and for any $w_{-n}^{-1} \in \A^n, n\geq 1$, $\beta(w_{-n}^{-1})=\sum_{j=0}^{n}\theta_j$. This implies that for $n\geq 1$, by denoting $s_n=\theta_{n}+\theta_{n+1}+\ldots,$ we have that
$$
\tilde{\rho_n}= \prod_{j=1}^{n+1}(1-s_{j}).
$$
By recalling Example \ref{ex:autoregressive}, we have that Theorem \ref{teo:algops2_1} holds if we suppose that
$$
\sum_{m=1}^{\infty}s_m< \infty ,
$$ 
and Theorem \ref{teo:algops2_2} holds if we suppose that there exists $m\in \{1,2,\ldots\}$ and $\epsilon >0$ such that
$$
s_n\leq \frac{(1-\epsilon)}{n}
$$
for all $n\geq m$.
\end{example}

\begin{remark} \label{remark:examplealgo2}
We consider the following generalization of the above example. Consider a connected and non-directed graph $(V,E)$, where $V$ is a finite set of vertices and $E$ is the set of edges.  For any $v\in V,$ let $E(v)$ be the set of vertices connected to $v$ by an edge. We consider the convention $v \in E(v)$.

Consider moreover a collection of non-negative numbers $\{\theta_j:j\geq 0\}$  such that $\sum_{j=0}^{\infty}\theta_j=1$. Suppose that $\theta_0>0$. Consider a model taking values in $\A=V$ and let
$$
\H=\{a_{-\infty}^{-1}\in \A^{\N}:a_{-n+1}\in E(a_{-n}), \text{ for all } n \geq 2\}.
$$
For any $x_{-\infty}^{-1}\in \H$, the probability kernel is given as follows. For any $w_{-\infty}^{-1}\in \H$ and for $g=w_{-1}$,
$$
p(g|w_{-\infty}^{-1})= \frac{1}{|E(w_{-1})|}\theta_0 +\sum_{j=1}^{\infty}\theta_j
\mathbf{1}_{E(w_{-1})^c\cup\{w_{-1}\}}(w_{-j}).
$$
For $g\in E(v)\setminus \{w_{-1}\}$,
$$
p(g|w_{-\infty}^{-1})= \frac{1}{|E(w_{-1})|}\theta_0
+\sum_{j=1}^{\infty}\theta_j\mathbf{1}_{\{w_{-j}\}}(g).
$$
Finally, we define $p(g|w_{-\infty}^{-1})=0$ for $g\not \in E(w_{-1})$. 

This kernel represents a process that can only jump to its neighbors on the graph $(V,E)$. However, the probability of jumping to each neighbor depends on the entire history of the process in a similar way to Example \ref{ex:algo2}. With assumptions similar to those made in Example \ref{ex:algo2}, we can use Algorithm \ref{algo:ps2} for this kernel with $\hat n =1$.
\end{remark}

\section{Proof of Theorems \ref{teo:cftp} and \ref{coro:cftp}} \label{sec:proofteocftp}

In the following $(U_n)_{n\in \Z}$ are i.i.d. uniform random variables in $[0,1)$. 
We will construct a function $F$ such that $X_n=F(U_n,U_{n-1}, \ldots)$ in such a way that $U_{n-1}, U_{n-2}, \ldots$ defines a partition of the interval $[0,1)$, associating each part of the interval to an element of $\A$, and $X_n$ will be the defined taking into account to which of these intervals $U_n$ belongs.


To define the coupling function $F(u_0,u_{-1},\ldots)$ for $u_0, u_{-1}, \ldots \in [0,1)$ we need to define the partition of the interval $[0,1)$ cited above. 
The intervals belonging to this partition are denoted
$$
I_0(g), \text{ for } g \in \A
$$
and
$$
I_m(g,  u_{-1},u_{-2},\ldots , u_{-m}) = I_m(g,  u_{-m}^{-1}) \text{, for } m\geq 1 \text{ and } g\in \A.
$$
These intervals are used to define the coupling function in the following way 
\begin{equation}\label{eq:I_m}
F(u_{-\infty}^{0})=\sum_{g\in \A}g \mathbf{1}\left\{u_0 \in \bigcup_{m=0}^{\infty}I_{m}(g,u_{-m}^{-1})\right\}+*\mathbf{1}\left\{u_0\in [0,1)\setminus \bigcup_{m=0}^{\infty}I_{m}(g,u_{-m}^{-1})\right\}.
\end{equation}
Here and in the following we use the convention $I_0(b,u_{0},u_{-1})=I_0(b)$. 

Algorithm \ref{algo:ps} uses the notion of ``temporary'' values for $X_m$, $m\in \Z$. In the perfect simulation algorithm we check if $X_m$ can be obtained only by using the values of $U_m,\ldots, U_{m-n}$, for $n\geq 0$. This is formalized by the following notation.
For any $n\geq 0$,
\begin{equation}\label{def:tempvalues}
F^{(n)}(u^0_{-n}):=\sum_{g\in \A}g \mathbf{1}\left\{u_0 \in \bigcup_{m=0}^{n}I_{m}(g,u^{-1}_{-m})\right\}+*\mathbf{1}\left\{u_0\in [0,1)\setminus \bigcup_{m=0}^{n}I_{m}(g,u^{-1}_{-m})\right\} . 
\end{equation}
Note that for any $u_{-\infty}^{0}\in[0,1)^{\N}$, $F(u^0_{-\infty})=\lim_{n\to\infty}F^{(n)}(u^0_{-n})$.

The intervals $
(I_n(g, u_{-n}^{-1}), n\geq 0, g\in \A) $  will be defined by means of a collection of positive numbers
$$
J_n(g,u_{-n}^{-1}): \A\times [0,1]^n \to [0,1]
$$ 
taking values in $ [0, 1]  $ in the following way:
for any $n\geq 0$ and $g\in \A$, 
\begin{equation}\label{eq:Ing}
    I_n(g, u_{-n}^{-1})= \left[\sum_{m=0}^{n-1}\sum_{b\in \A}J_m(b, u_{-m}^{-1})
 +\sum_{b< g}J_n(b,u_{-n}^{-1}), \sum_{m=0}^{n-1}\sum_{b\in \A}J_m(b,u_{-m}^{-1})
 +\sum_{b\leq g}J_n(b,u_{-n}^{-1})\right).
\end{equation}
Note that $|I_n(g,u_{-n}^{-1})|=J_n(g,u_{-n}^{-1})$. Note also that for the case $n=0$, the definition above reads
$$
I_0(g)= \left[\sum_{b< g}J_0(b), \sum_{b\leq g}J_0(b)\right).
$$

Before giving the definition of $(J_n(g,  u_{-n}^{-1}), n\geq 0, g\in \A)$, let us stress that since all $ J_n(g,u_{-n}^{-1}) $ are non negative (as will be proved in Proposition \ref{prop:monotone}), in \eqref{eq:I_m}, for any given sequence $u_0,u_{-1}, \ldots,$ there exist at most one index $ m \geq 0 $ and one value $g$ such that the condition $u_0 \in I_{m}(g,u_{-m}^{-1})$ is fulfilled.  
Therefore, $I_n(g,u_{-n}^{-1})$ is a right-open sub-interval of $[0,1)$ (including the empty set) for any $g\in \A$.

In what follows, for any $j \in \Z,$ we shall also consider shifted versions $F^{(n)}(u^j_{j-n})$ of $F^{(n)}(u^0_{-n}) ,$ $ I_n(g, u_{j-n}^{j-1})$ of $ I_n(g, u_{-n}^{-1}) $ and $J_n(g, u_{j-n}^{j-1})$ of $ J_n(g, u_{-n}^{-1}),$ which are obtained by applying the respective definitions of $F^{(n)}, I_n $ and $ J_n$ to $ u^j_{j-n} $ rather than to $ u^0_{-n}.$

Let us now give the definition of the successive values of $ J_n (g, U_{- n }^{-1}) .$
We take $J_0(g)\mydef\alpha(g)$ and define recursively for any $n\geq 1$,
\begin{multline} \label{eq:jngalternativeform}
J_{n}(g,u_{-n}^{-1})\mydef
\alpha(g|F^{(n-1)}(u_{-n}^{-1})F^{(n-2)}(u_{-n}^{-2})\ldots F^{(0)}(u_{-n}))-\\\alpha(g|F^{(n-2)}(u_{-(n-1)}^{-1})F^{(n-3)}(u_{-(n-1)}^{-2})\ldots F^{(0)}(u_{-(n-1)})) .
\end{multline}
Note that by construction, 
$$
J_{n}(g,u_{-n}^{-1})=\alpha(g|F^{(n-1)}(u_{-n}^{-1})F^{(n-2)}(u_{-n}^{-2})\ldots F^{(0)}(u_{-n}))-\sum_{m=0}^{n-1}J_{m}(g,u_{-m}^{-1}).
$$

\begin{remark}\label{remark:explainJng} 
Recalling Algorithm \ref{algo:ps}, the term defined in \eqref{eq:jngalternativeform} satisfies for any $g\in \A$ and for any $n\geq 1$,
$$
J_{n}(g,U^{-1}_{-n})=\alpha(g|X_{-1}^{(-n)}\ldots X_{-n}^{(-n)})-\alpha(g|X_{-1}^{-(n-1)}\ldots X_{-(n-1)}^{-(n-1)}).
$$  
The positive random variable $J_{n}(g,U^{-1}_{-n})$ quantifies how much the minimal probability of having the symbol $g$ at time $0$ increases when we consider the temporary values $X_{-1}^{(-n)},\ldots, X_{-n}^{(-n)}$,  which depend on the random variables $U_{-n}^{-1}$, instead of considering $X_{-1}^{-(n-1)},\ldots, X_{-(n-1)}^{-(n-1)}$, which depend on the random variables $U_{-(n-1)}^{-1}$. In other words, it quantifies how much the minimal probability of observing the symbol $g$ at time $0$ increases when we additionally consider the uniform random variable $U_{-n}$. 
\end{remark}

\begin{remark} The sequential construction of the functions $ F^{(n)} (u_{-n}^0)$ given in (\ref{def:tempvalues}) above is well defined. Indeed, to obtain $ F^{(n)} (u_{-n}^0), $ one only needs to know all 
intervals $ I_m ( g, u^{- 1}_{-m}) , m \le n .$ And in turn, to define these intervals through their respective lengths in (\ref{eq:Ing}) and (\ref{eq:jngalternativeform}), we see that we only rely on previously chosen functions $ F^{(n-1)}, \ldots, F^{(0)}, $ which are applied to substrings of $ u_{-n}^{-1}.$ 
\end{remark}

The following proposition is a direct consequence of the definitions above. 

\begin{proposition}\label{prop:monotone} 
Let $u_{-\infty}^0\in [0,1)^{+\infty}$. Suppose there exists $ n \geq 0 $ such that $ F^{(n)}(u_{-n}^0) = b$ for some $b \in \A.$ Then $ F^{(n+k)}(u_{-(n+k)}^0) = b$ for all $k\geq 1$. In particular, $J_n (g) \geq 0$ for all $ n \geq 0$ and $g\in \A$.
\end{proposition}
\begin{proof}
For the first assertion, just note that for any $k\geq 1$,
\begin{multline*}
\{F^{(n)}(u_{-n}^0) = b\}=\left\{u_{0} \in \bigcup_{j=0}^{n}I_{j}(b,u_{-j}^{-1})\right\} \subset \\ \left\{u_{0} \in \bigcup_{j=0}^{n+k}I_{j}(b,u_{-j}^{-1})\right\}=\{F^{(n+k)}(u_{-(n+k)}^0) = b\}.
\end{multline*}
The second assertion follows from \eqref{eq:jngalternativeform}
together with the monotonicity properties stated in Remark \ref{rem:monotone} above.
\end{proof}

Now we can prove Theorem \ref{teo:cftp}. 

\begin{proof}
The proof follows directly by the construction above by noticing that
for any $m\in \Z$ and for any $k\geq 0$,
\begin{equation*} 
X_{m}^{(m-k)}= F^{(k)}(U^m_{m-k})
\end{equation*}
and that 
$$ X_m = \lim_{n\to\infty}F^{(n)} ( U^m_{m-n})=F^{(T[m])} ( U^m_{T[m]} ).$$ 
From this the assertion follows. 
\end{proof}

The proof of Theorem \ref{coro:cftp} is analogous to the proof of Theorem \ref{teo:cftp} and therefore omitted. The construction of the functions $\tilde{F}^{(j)}: [0,1)^{\N} \to \A, \text{ for } j=0,\ldots, n_0-1,$ are analogous to the construction of the function $F$ above. Having constructed these functions, the proof of Theorem \ref{coro:cftp} follows as the the proof of Theorem \ref{teo:cftp} above.

\section{Proof of Theorems \ref{teo:infinitepath}, \ref{teo:renewal}, \ref{teo:algops2_1} and \ref{teo:algops2_2}} \label{sec:proofsalgo1}

Recall that according to Definition \ref{def:algo1}, for $k\geq 0,$ $X_0^{(-k)},\ldots,X_{-k }^{(-k)} 
\in \A_*$ are the values of $X_0,\ldots,X_{-k}$ after the first $k+1$ iterations of Algorithm \ref{algo:ps}, that is, based on the simulation of $ U_{-k }^0.$  For any $n \geq 0$,
\begin{multline} \label{eq:probt0geqn}
\P(|T[0]|>n)= 
\\
\sum_{a_{-n}^{-1}\in \A_*^n}\P(X_{-n}^{(-n)}=a_{-n})
\ldots \P(X_{-1}^{(-n)}=a_{-1}|(X^{(-n)})_{-n}^{-2}=a_{-n}^{-2})\P(X_{0}^{(-n)}=* |(X^{(-n)})_{-n}^{-1}=a_{-n}^{-1})
\\
=\sum_{a_{-n}^{-1}\in \A_*^n}\alpha(a_{-n})\alpha(a_{-(n-1)}|a_{-n})\ldots \alpha(a_{-1}|a_{-n}^{-2})\alpha(*|a_{-n}^{-1}).
\end{multline}
Note that in the case $n=0$ this reads
$$
\P(|T[0]|>0)=\alpha(*).
$$
In particular, 
$$
\E(|T[0]|)=\sum_{n=0}^{+\infty}\sum_{a_{-n}^{-1}\in \A_*^n}\alpha(a_{-n})\alpha(a_{-(n-1)}|a_{-n})\ldots \alpha(a_{-1}|a_{-n}^{-2})\alpha(*|a_{-n}^{-1}).
$$

We interpret the above formula in terms of an auxiliary chain 
$(Y_n)_{n\geq 0}$, 
assuming values in $\A_*$ and having increasing memory, which we introduce now. We start from 
$\P(Y_0=g)=\alpha(g)$ and then sequentially define 
$$
\P(Y_n=g|Y_0^{n-1}=a_0^{n-1})=\alpha(g|a_0^{n-1}),
$$
for any $g\in \A_*$ and $a_0^{n-1} \in \A_*^n$. 

\begin{remark}
The purpose of introducing this auxiliary chain is to compare the probability of events related to $(Y_n)_{n\geq 0}$ with the probability of events related to $(T[n]:n\in \Z)$. 

Note that for any $k\in \Z$ and $n\in \N$,
$$
Y_n \sim X_{k}^{(k-n)}.
$$
In particular,
$$
Y_n \sim X_{n}^{(0)}, \text{ for any } n\in \N.
$$
Therefore, $(Y_n)_{n\geq 0}$ has the law of a `forward in time' version of Algorithm \ref{algo:ps}, which instead of starting from a certain time and going backward in time, updating the values of the chain for each new information obtained, goes forward in time checking which values of the chain can be obtained from a certain time on. 
\end{remark}

We have that
\begin{equation}\label{eq:pn}
p_n:=\P(|T[0]|>n)=\P(Y_n=*)
\end{equation}
and
\begin{equation}\label{eq:ET0}
\E(|T[0]|)=\sum_{n\geq 0} \P(Y_n=*)=\E\left(\sum_{n\geq 0} \mathbf{1}\{Y_n=*\}\right).
\end{equation}
Moreover,  
$$
\P(|T[0,n]|>m) \leq \sum_{j=0}^n\P(|T[j]|>m)=\sum_{j=0}^n\P(|T[0]|>m+j).
$$
Therefore,
\begin{equation} \label{eq:t0nbound}
\P(|T[0,n]|>m)\leq \sum_{k=m}^{m+n}\P(Y_k=*)
\end{equation}
and
\begin{equation} \label{eq:t0n}
\E(|T[0,n]|) \leq \sum_{m\geq 0} \sum_{k=m}^{m+n}\P(Y_k=*)\leq (n+1) \E\left(\sum_{n\geq 0} \mathbf{1}\{Y_n=*\}\right).
\end{equation}

Now we can prove Theorem \ref{teo:infinitepath}.

\begin{proof}
Let
$$
\eta_1= \inf\{m\geq 0:Y_m=* \}
$$
with the convention $\inf \emptyset =\infty,$ and for $n\geq 1$, let
$$
\eta_{n+1}= \inf\{m\geq \eta_n+1:Y_m=*\},
$$
with the convention $\eta_{n+1}=\eta_{n}$ in the case $\eta_{n} =\infty$. Notice that 
\begin{equation}\label{eq:eta_1andrho}
\P ( \eta_1 \geq n ) =\P\left(\bigcap_{k=0}^{n-1}\bigcup_{a_k\in \A}\{Y_k=a_k\}\right) =  \sum_{a_{0}^{n-1}\in \A^n}\alpha(a_{0})\alpha(a_{1}|a_{0})\ldots \alpha(a_{n-1}|a_{0}^{n-2}) = \rho_n.
\end{equation}
By assumption, 
$$
\P(\eta_1 < \infty)=1- \lim_{n\to \infty}\P(\eta_1 \geq n) =1- \lim_{n\to \infty}\rho_n 
\leq 1-c.
$$
Now we introduce
$$ \Omega_n := \bigcup_{j=n-1}^{\infty}\left\{w_{1}^{j} \in  A_*^j: w_{j}=*, \sum_{k=1}^{j-1}\mathbf{1}\{w_k=*\}=n-2\right\} ,$$ 
the set of all possible past sequences corresponding to the event $\bigcap_{m=1}^{n-1}\{\eta_m<\infty\}.$ Then, 
\begin{multline*}
\P\left(\eta_n< \infty \Big|\bigcap_{m=1}^{n-1}\{\eta_m< \infty\}\right)=
\\
\sum_{w \in \Omega_n }\P(w)\left[ 1-\lim_{n\to \infty}\sum_{a_{0}^{n-1}\in \A^n}\alpha(a_{0}|w)\alpha(a_{1}|a_{0}w)\ldots \alpha(a_{n-1}|a_{0}^{n-2}w)\right].
\end{multline*}
By monotonicity (recall Remark \ref{rem:monotone}), we have that the term above is upper bounded by
$$
\sum_{w}\P(w)\left[1-\lim_{n\to \infty}\sum_{a_{0}^{n-1}\in \A^n}\alpha(a_{0})\alpha(a_{1}|a_{0})\ldots \alpha(a_{n-1}|a_{0}^{n-2})\right]\leq 1-c.
$$
This implies that
$$
\P(\eta_n < \infty)=\P(\eta_j < \infty,j=1,\ldots, n) \leq (1-c)^n.
$$
To conclude part 1 of Theorem \ref{teo:infinitepath} note that, since by (\ref{eq:ET0}),  $\E(|T[0]|)$ is the expected number of visits of the auxiliary chain to $ *,$ we have that 
$$
\E(|T[0]|)=\E\left(\sum_{n\geq 0} \mathbf{1}\{Y_n=*\}\right)=\sum_{m\geq 1 }\E\left( \mathbf{1}\{\eta_m < \infty \}\right) \le  \sum_{m\geq 1}(1-c)^m=\frac{1-c}{c}.
$$
The proof of part $2$ of Theorem \ref{teo:infinitepath} follows directly from \eqref{eq:t0n}.
To prove part $3$ of Theorem \ref{teo:infinitepath}, 
note that for any $n \geq 0$, 
\begin{multline*}
\P(|T[0,+\infty) |>n)= \\
\sum_{a_{-n}^{-1}\in \A_*^n}\P(X_{-n}^{(-n)}=a_{-n})
\cdots \P(X_{-1}^{(-n)}=a_{-1}|(X^{(-n)})_{-n}^{-2}=a_{-n}^{-2}) \times \\
\times \P\left(\bigcup_{k=0}^{\infty}\{X_{k}^{(-n)}=*\} |(X^{(-n)})_{-n}^{-1}=a_{-n}^{-1}\right).
\end{multline*}
By noticing that
\begin{multline*}
\P\left(\bigcup_{k=0}^{\infty}\{X_{k}^{(-n)}=*\} |(X^{(-n)})_{-n}^{-1}=a_{-n}^{-1}\right)= \\ 
\sum_{k=0}^{\infty}\P\left(\{X_{k}^{(-n)}=*,X_{j}^{(-n)}\neq *, j=0,\ldots,k-1\} |(X^{(-n)})_{-n}^{-1}=a_{-n}^{-1}\right),
\end{multline*}
we conclude that
\begin{multline}\label{eq:inter}
\P(|T[0,+\infty)|>n)=\\
\sum_{a_{-n}^{-1}\in \A_*^n}\sum_{k=0}^{\infty}\sum_{b_{0}^{k-1}\in \A^k}
\alpha(a_{-n})\alpha(a_{-(n-1)}|a_{-n})\cdots \alpha(a_{-1}|a_{-n}^{-2})\alpha(b_0|a_{-n}^{-1})\\
\cdots \alpha(b_{k-1}|b_0^{k-2}a_{-n}^{-1})\alpha(*|b_0^{k-1}a_{-n}^{-1}) . 
\end{multline}
Let
$$
\bar{\eta}=\sup\{m\geq 0:Y_m=*\},
$$
with the convention $\sup \emptyset =-\infty$. Interpreting (\ref{eq:inter}) in terms of the auxiliary chain, we see that 
$$
\P(|T[0,+\infty)|>n)=\P(\bar{\eta} \geq n).
$$
Now, note that
$$
\left\{\bar{\eta}>n, \text{ for all } n=1,2,\ldots \right\}=\left\{\eta_j<\infty, \text{ for all } j=1,2,\ldots \right\},
$$
such that
$$
\lim_{n\to \infty}\P(|T[0,+\infty]|>n)=\lim_{n\to \infty}\P(\bar{\eta} >n) =\lim_{n\to \infty}\P(\eta_n <\infty)=0.
$$

\end{proof}

Now we prove Theorem \ref{teo:renewal}.

\begin{proof}
Let us consider a representation of the process $(Y_n)_{n=0,1,\ldots}$ using a sequence of i.i.d. uniform random variables $(V_n)_{n=1,2,\ldots}$ assuming values in $[0,1)$. Let
$$
Y_n=
\begin{cases}
g &\text{, if } V_n \in \left[\sum_{b<g}\alpha(b|Y_0^{n-1}),\sum_{b\leq g}\alpha(b|Y_0^{n-1})\right), \text{ for } g\in \A, \\
* &\text{, in the contrary case}.
\end{cases}
$$
Now, we define 
$$
\tilde{Y}_n=
\begin{cases}
g &\text{, if } V_n \in \left[\sum_{b<g}\alpha(b|Y_{0}^{n-1}),\sum_{b<g}\alpha(b|Y_{0}^{n-1})+ \alpha(g|Y_{L_{n-1}+1}^{n-1})\right), \text{ for } g\in \A, \\
* &\text{, in the contrary case},
\end{cases}
$$
with $L_{n-1}=\sup\{m=0,\ldots,n-1: \tilde Y_m=*\}$. Here, we use the convention $\sup\emptyset=0$. Note that, by Remark \ref{rem:monotone}, we have that $\alpha(g|Y_{L_{n-1}+1}^{n-1})\leq \alpha(g|Y_{0}^{n-1})$, for any $Y_{0}^{n-1} \in \A_*^n,$ and therefore the process $(\tilde{Y}_n)_{n=0,1,\ldots}$ is well defined. Moreover, for all $n,$ if $ \tilde Y_n = g, $ then also $ Y_n = g.$ So we have that $\tilde{Y}_n \in \{Y_n,*\}$, for any $n \geq 0,$ and the two processes coincide during excursions of $ \tilde Y$ out of $*.$ Therefore, for any $m \geq 0, $
$$
\sum_{n=0}^m \mathbf{1}\{Y_n=*\} \leq \sum_{n=0}^m \mathbf{1}\{\tilde{Y}_n=*\}.
$$
Moreover,
\begin{equation}\label{eq:YandYtilde}
p_m=\P(Y_m=*)\leq \P(\tilde{Y}_m=*).
\end{equation}

Now, considering $\tau_0:=0$ and $\tau_n:=\inf\{m>\tau_{n-1}:\tilde{Y}_m=*\}$, we have that $(\tau_m)_{m\geq 1}$ are the marks of a renewal process, and the excursions $( \tilde Y_{ \tau_n + k })_{ 0 \le k < \tau_{n+1} } , n \geq 1, $ are i.i.d. In particular, standard renewal arguments imply that we have almost sure convergence 
$$ \frac1n \sum_{k=1}^n  \mathbf{1}\{\tilde{Y}_k=*\} \to \frac{1}{\E ( \tau_1)},$$
and thus, by dominated convergence, also 
$$ \frac1n \sum_{k=1}^n  \P(\tilde{Y}_k=*) \to  \frac{1}{\E ( \tau_1)}.$$ 
But, by assumption,
$$
\E(\tau_1)=\sum_{m=1}^{+\infty} \P(\tau_1\geq m)=\sum_{m=1}^{+\infty} \rho_m=+\infty ,
$$
such that 
$$ \frac1n \sum_{k=1}^n  \P(\tilde{Y}_k=*) \to  0.$$ 
We now show that this implies necessarily that $$
\lim_{n\to \infty}\P({Y}_n=*) = \lim_n p_n = 0.
$$
Indeed, by (\ref{eq:pn}), the sequence $ (p_n)_n$ is a non-increasing sequence of positive numbers which thus converges. Suppose that $ \lim_n p_n > 0.$ Then $ \liminf_n \P(\tilde{Y}_n=*) > 0 $ as well and then there exist $ c_1, n_1 > 0 $ such that for all $ n \geq n_1, $ $ \P(\tilde{Y}_n=*) \geq c_1 > 0.$ This is in contradiction with $ \frac1n \sum_{k=1}^n  \P(\tilde{Y}_k=*) \to  0.$ 

Together with \eqref{eq:YandYtilde} and \eqref{eq:t0nbound}, this concludes the proof.
\end{proof}

The proofs of Theorems \ref{teo:algops2_1} and \ref{teo:algops2_2} are analogous to the proofs of Theorems \ref{teo:infinitepath} and \ref{teo:renewal}, and in the following we just sketch the main ideas.

\begin{proof}
Let $(V_n^{a_{- \hat n}^{-1}}:n\geq0,a_{- \hat n}^{-1}\in \mathcal{C})$ be i.i.d. uniform random variables in $[0,1)$. For any $n\geq 0$ and $a_{- \hat n}^{-1}\in \mathcal{C}$, the process $(Z_n^{a_{-\hat n}^{-1}})_{n\geq 0}$ is defined as follows. In the following, let us consider the convention $\inf\emptyset=*$. For $k\geq0$ and $n=kn_0,\ldots,kn_0-1$, we have
\begin{equation*} 
Z_n^{a_{-\hat{n}}^{-1}}=\inf\left\{g\in \A: V_n^{a_{- \hat n}^{-1}}< \sum_{b\leq g}\alpha(g|(Z^{a_{-\hat{n}}^{-1}})_{kn_0}^{n-1}a_{-\hat{n}}^{-1})\right\}.
\end{equation*}
Now, we define the process $(Z_n)_{n\geq 0}$ as follows. For $n=kn_0,\ldots,kn_0-1$, we have
$$
Z_n=
\begin{cases}
b, &\text{ if } Z_n^{a_{-\hat{n}}^{-1}}=b, \text{ for all } a_{-\hat{n}}^{-1} \in \mathcal{C} \text{ and some } b\in \A,\\
*, &\text{ in the contrary case}.
\end{cases}
$$
For $k\geq 1$ and $n=kn_0,\ldots,kn_0-1$, we have two cases. If $Z_{kn_0-\hat n}^{kn_0-1}=b_{-\hat n}^{-1}$, for some $b_{-\hat n}^{-1} \in \mathcal{C}$, we have
\begin{equation*} 
Z_n=\inf\left\{g\in \A: V_n^{b_{- \hat n}^{-1}}< \sum_{b\leq g}\alpha(g|Z_{0}^{n-1})\right\}.
\end{equation*}
In the contrary case, we have
$$
Z_n=
\begin{cases}
b, &\text{ if } Z_n^{a_{-\hat{n}}^{-1}}=b, \text{ for all } a_{-\hat{n}}^{-1} \in \mathcal{C} \text{ and some } b\in \A,\\
*, &\text{ in the contrary case}.
\end{cases}
$$

We have that
$$
\P(|\tilde{T}[0]|>n)=\P(Y_{nn_0}=*)
$$
and
$$
\E(|T[0]|)=\sum_{n\geq 0} \P(Y_{nn_0}=*)=\E\left(\sum_{n\geq 0} \mathbf{1}\{Y_{nn_0}=*\}\right).
$$
By proceeding in a similar way as in the proofs Theorems \ref{teo:infinitepath} and \ref{teo:renewal}, the assertions of Theorems \ref{teo:algops2_1} and \ref{teo:algops2_2} follow. 
\end{proof}

\section{Discussion} \label{sec:dicussion}
In this section we will compare our results to algorithms and results presented in other articles. In particular we will show that the results presented in these articles cannot be applied to the examples presented in Section \ref{sec:examples}.
Finally we will also discuss some limitations of our approach. 
 \subsection{Comparing to other results}
Our framework is similar to the one presented in \cite{desantispiccioni}, and the present article is very much inspired by this work.  For a probability kernel $p$, \cite{desantispiccioni} consider a coupling function $F:[0,1]\times \mathcal{H} \to \A$ such that
$$
F(u_0,x_{-1},x_{-2}, \ldots)=\sum_{g\in \A}g \mathbf{1}\left\{u_0 \in \bigcup_{m=0}^{-\infty}I_{m}(g,x_{-1},\ldots, x_{m})\right\},
$$
where $\mathcal{H} \subset \A^{\N}$ is the set of admissible histories. This means that each possible past defines a partition of the interval $[0,1]$. The authors study the backward coalescence time $\tau_0(U_{-\infty}^{0}):$  $-\tau_0$ is a stopping time, and $X_0$ can be defined from $U_0,U_{-1}, \ldots, U_{\tau_0}.$ The existence of such a backward coalescence time and its finiteness directly induce a perfect simulation algorithm. 

Theorem 1 of \cite{desantispiccioni} deals with perfect simulation algorithms having backward coalescence times that are based on {\it information depths}. A necessary condition for this theorem to work is that almost surely
$$
\lim_{h\to \infty}A_h(U_{-h}^{-1}) =1, 
$$
where the information depth $A_h(U_{-h}^{-1})$ is given by
$$
A_h(U_{-h}^{-1})=\inf\{\beta(w_{-h}^{-1}):w_{-\infty}^{-1}\in \H:w_{-k}=g \text{ if } U_{-k}\in I_0(g), k\leq h\}.
$$

Let us compare this approach with our Example \ref{ex:imitation}. Notice that in this example, 
$$
\beta(*a_{-k}^{-1})=c,
$$
for any $a_{-k}^{-1} \in \A^k$. Therefore, $A_h(U_{-h}^{-1})=c$, whenever $U_{-1}>c$. So $\lim_{h\to \infty}\P(A_h(U_{-h}^{-1})<1) \geq 1-c$, such that Theorem 1 of \cite{desantispiccioni} cannot be applied in the situation of this example.

In the framework of Remark \ref{rem:imitationexample}, assuming additionally that
$$
\inf_m f_m(k)=0,
$$
for each $k\geq 1$, then 
$$
\beta(*a_{-k}^{-1})=c,
$$
for any $a_{-k}^{-1} \in \A^k$.  Again Theorem 1 of \cite{desantispiccioni} cannot be applied here.

Finally, in Example \ref{ex:other},
by supposing that for any $k\geq 1$, 
$$
\lim_{j\to \infty}\inf\{q_m(k):m > j\}= 0,
$$
we have again that 
$$
\beta(*a_{-k}^{-1})\leq c+(1-c)\sum_{k=1}^{\infty}\inf_{m \geq \sum_{j=1}^{k} a_{-j}}q_m(k)=c,
$$
for any $a_{-k}^{-1} \in \A^k$. 




\cite{gallo} presents a perfect simulation algorithm for chains with memory of variable length. Considering a countable alphabet $A$ and a tree $\tau \subset \{a_{-n}^0: a_{-n}^0 \in A^n, n\in \{0,1,\ldots\}\cup\{\infty\}\}$, they propose a perfect simulation algorithm to obtain a sample of the stationary chain compatible with the probabilistic context tree $(\tau,p),$ where $p=(p(g|w):g\in \A, w\in \tau)$. Note that probabilistic context trees are special cases of chains with memory of infinite length. 

\cite{gallo} constructs a coupling function $F:[0,1] \times \tau \to [0,1]$. In a framework similar to the one presented in \cite{cff} and \cite{desantispiccioni},  each context defines a partition in the interval $[0,1]$. A perfect simulation algorithm and a visible regeneration scheme are obtained imposing conditions on $(l(n): n\in \N)$. Examples \ref{ex:imitation} and \ref{ex:other} considered in this article are in fact chains with memory of variable length however, the results presented by \cite{gallo} can not be applied to these examples.

The perfect simulation algorithm introduced by \cite{ciftp} uses the \textit{canonical coupling function} to propose a construction that follows the idea of the original coupling from the past algorithm in \cite{cftp}. At each step of the algorithm, after sampling the uniform random variable that is used to obtain the value of the chain at a certain instant, the algorithm checks all possible values of the chain at that instant for all the possible past infinite strings. \cite{ciftp} present kernels in which this algorithm can be applied even in the case $\beta=0$. This article focuses on algorithmic aspects and does not provide a deep study on the conditions guaranteeing that the algorithm stops almost surely in a finite time.

The perfect simulation algorithm introduced by \cite{gallogarcia} uses the \textit{canonical coupling function} to propose a construction that deals with chains with memory of infinite length which are locally continuous with respect to a context tree. In particular, this perfect simulation algorithm assumes $\beta>0$. Our construction is obviously different from this one since it needs a different coupling function than the canonical one. The comparison between our results in the cases $\beta>0$ and the results presented by \cite{gallogarcia} is difficult.

\cite{gallotakahashi} consider regular probability kernels $p$ on a finite alphabet $A$. This means that $p$ is strongly non-null ($\inf_{g}\alpha(g)>0$) and continuous. Moreover, the considered kernel is attractive, which means that $p$ satisfies some monotonicity properties. Note that our results and examples do not require any of the assumption above.

\cite{gallotakahashi} also present a concentration of measure result at exponential rate that holds for any process assuming values in a finite alphabet admitting a coupling from the past algorithm. We can use this result in our framework (see Proposition \ref{prop:gallotakahashi}). With the additional assumption that the kernel is regular, Theorem 4 of \cite{gallotakahashi} proves that the concentration of measure at exponential rate implies the uniqueness.

\subsection{Limitations of our method}
We continue this discussion with an example showing that there exist processes that cannot have a perfect simulation using our algorithms even though they are ergodic and possess a unique translation-invariant stationary measure.

\begin{example} \label{ex:cantuseouralgorithms}

Consider $\A=\{0,1\}$ and let $r_n, n\geq 2$  be an increasing sequence of positive numbers such that $r_n\to 1$ as $n\to \infty$. Let $\mathbf{0,1} $ be the histories formed by all zero and all one, respectively. Let $\H= \{0,1\}^\N \setminus \{ \mathbf {0,1} \}$.
For any $x_{-\infty}^{-1}\in S$ and for any $g\in \A$, let
\begin{equation}
    p(g | x^{-1}_{-\infty})=
r_{\tau(x^{-1}_{-\infty})}\mathbf{1}_{\{x_{-1}\}}(g)+(1-r_{\tau(x^{-1}_{-\infty})})\mathbf{1}_{\{x_{-1}\}^c}(g),
 \label{esempio2}
\end{equation}
where
\[
\tau(x^{-1}_{-\infty}) = \sup\{ k\in \N : x_{-1} = \cdots = x_{-k}\}.
\]
Let us assume that
$$
\sum_{n=2}^{\infty}\prod_{j=2}^nr_j<\infty.
$$
First, note that $\beta=0$. Moreover, for any $n\geq 1$ and for any $b\in \A$, we have that if $a_{-j}=b, \text{ for any } j=1,\ldots,n$,
$$
\alpha(g|a^{-1}_{-n})=0,
$$
for $g\neq b$. Therefore, for any $n \geq 1$ the Markov kernel of order $n$
defined in Definition \ref{def:ptildeandpmarkov} has the states $\mathbf{0}^{-1}_{-n}$ and $\mathbf{1}^{-1}_{-n}$ as absorbing states and therefore, once more, the condition of possessing a unique closed class is violated. 
However, we know that the process is ergodic and that there exists a translation-invariant stationary measure of the process. In fact, just note that the time in which the chain remains in a single symbol before jumping to the other symbol has finite expectation.
\end{example}
\subsection{Comparing our coupling function to the ``canonical'' one
}\label{sec:cc}
The main difference between the framework presented in \cite{cff} and \cite{desantispiccioni} and the one introduced in Section \ref{sec:proofteocftp} for Algorithm \ref{algo:ps} is the fact that our coupling function is defined on the set $[0,1]^{+\infty}$ instead of $[0,1]\times \H$. In other words, the coupling function depends on the uniform random variables instead of depending on a string of symbols of $\A$. However, the process $(X_n)_{n\in \Z}$ that we defined using the coupling function $F$ in Section \ref{sec:proofteocftp} can be constructed using a coupling function ``of the same type'' as those considered in 
\cite{cff} or in \cite{desantispiccioni}. This implies that our construction here is more general. 
In what follows, we explain exactly what we mean by ``of the same type''.

The coupling function proposed by \cite{cff} (which is called \textit{canonical coupling function} by \cite{desantispiccioni}) has the form
\begin{equation}\label{eq:canonical}
G(u,x_{-\infty}^{-1}):=\sum_{g\in \A}g\sum_{n=0}^{+\infty}\mathbf{1}\{u \in \tilde{I}_n(g|x_{-n}^{-1})\},
\end{equation}
where $(\tilde{I}_n(g|x_{-n}^{-1}):g\in \A, n=0,1,\ldots)$ are right-open intervals. By assuming that 
\begin{equation}\label{eq:continuitycanonical}
\alpha(g|x_{-n}^{-1})\to p(g|x_{-\infty}^{-1}), \text{ as } n \to \infty, \text{ for any } g\in \A \text{ and } x_{-\infty}^{-1}\in \H,
\end{equation}
we have that $(\tilde{I}_n(g|x_{-n}^{-1}):g\in \A, n=0,1,\ldots)$ is a partition of the interval $[0,1)$, where the first intervals in the left are $\tilde{I}_0(g)$, sequentially for $g\in \A$, satisfying $|\tilde{I}_0(g)|=\alpha(g)$, and the next intervals are $I_1(g|x_{-1})$, sequentially for $g\in \A$, satisfying $|\tilde{I}_1(g|x_{-1})|=\alpha(g|x_{-1})-\alpha(g)$, and so on. If the condition in \eqref{eq:continuitycanonical} is not satisfied, we additionally define $G(u,x_{-\infty}^{-1})=*$ for
$$
u\in [0,1)\setminus\bigcup_{n=0}^{\infty}\bigcup_{g\in \A}\tilde{I}_n(g|x_{-n}^{-1}).
$$

This coupling function follows the natural order of considering the $n$-th first symbols of the past fixed successively for $n=0,1,2,\ldots$. This idea can be generalized by considering that we successively fix the symbols of the past in a different way. For a sequence of sets $\tilde{S}_0 \subseteq \tilde{S}_1 \subseteq \ldots$, where $\tilde{S}_n \subseteq \{1,\ldots, n\}$ for any $n=0,1,\ldots$, and for $x_{-\infty}^{-1}\in \H$, let
$$
x_{-k}^{(\tilde{S}_n)} =
\begin{cases}
x_{-k}, &\text{ if } k \in \tilde{S}_n, \\
*, &\text{ if } k \not\in \tilde{S}_n.
\end{cases}
$$
In this framework we may construct a coupling function $G^{(\tilde{S}_n)_{n\geq 0}}:[0,1]\times \H \to \A$, similar to \eqref{eq:canonical}, in which the first intervals in the left are $\tilde{I}^{(\tilde{S}_0)}_0(g)$, sequentially for $g\in \A$, satisfying $|\tilde{I}_0^{(\tilde{S}_0)}(g)|=\alpha(g)$, and the next intervals are $I_1^{(\tilde{S}_1)}(g|x_{-1}^{(1)})$, sequentially for $g\in \A$, satisfying $|\tilde{I}_1^{(\tilde{S}_1)}(g|x_{-1}^{(1)})|=\alpha(g|x_{-1}^{(\tilde{S}_1)})-\alpha(g)$, and so on. This means that we sequentially choose to fix symbols in the past based on the increasing sequence of sets $\tilde{S}_1, \tilde{S}_2, \ldots$.

Coming back to our Algorithm \ref{algo:ps}, for any $m\in \Z$ and $n\geq 1$, let $S_n(U_{m-n}^{m-1})=\{j\in \{1,\ldots,n\}: X^{(m-n)}_{m-j}\neq *\}$. Moreover, let $S_0=\emptyset$.
It follows that $(X_n)_{n\in \Z}$ satisfies
$$
X_n=G^{(S_j(U_{-n-j}^{-n-1}))_{j\geq 0}}(U_n,X_{-\infty}^{n-1}),
$$
where the increasing sequence that we consider is random. Therefore, the above construction shows that the coupling function introduced in Section \ref{sec:proofteocftp} for Algorithm \ref{algo:ps} can be interpreted as a modification of the \textit{canonical coupling function} which instead of successively considering the $n$-th first symbols of the past fixed successively for $n=0,1,2,\ldots$, consider that the part of the past to be fixed is random, depending only on $(U_n)_{n\in \Z}$. The part of the past to be fixed is exactly the spontaneous symbols and those obtained by the update procedure introduced in Algorithm \ref{algo:ps}.

\subsection{The set of admissible histories}

The set $\H$ (see Definition \ref{def:kernel_and_adimissible_histories}) contains all the  histories that are present in the stationary regime as proved in Proposition \ref{prop:admissiblehistories}. However, not all elements of $\H$ are present in the stationary regime of the process, as illustrated by the following example.

\begin{example}
Consider the Markov chain assuming values in $\A=\{1,2,3\}$ with transition probabilities given by the matrix $p:\A\times\A\to [0,1]$ defined as follows
$$
p(1|2)=p(2|1)=p(1|1)=p(2|2)=\frac{1}{2}, \quad  p(3|3)=0.9, \text{ and } p(1|3)=0.1.
$$
In this example, the symbol $3$ is transient, and therefore, this symbol has no mass in the stationary measure of the process. However, we have that
$3^{\N} \in \H.$
\end{example}

The following example illustrates the importance of the definition of the set of admissible histories. 
We present a kernel in which the assumptions of our results do not hold when we consider a set greater than $\H$ instead of $\H$.



\begin{example} \label{ex:admissiblehistories}

Consider $\A=\{0,1,2\}$ and let $r_n, n\geq 2$  be an increasing sequence of positive numbers such that $r_n\to 1$ as $n\to \infty$. Denote $\mathbf{0,1,2} $ as the histories formed by all zero, all one and all two, respectively. Let $S= \{0,1,2\}^\N \setminus \{ \mathbf {0,1,2} \}$.
For any $x_{-\infty}^{-1}\in S$ and for any $g\in \A$, let
\begin{equation}
    p (g | x^{-1}_{-\infty})=
\begin{cases}
r_n &\text{ if } g=x_{-1} \text{ and } x_{-1} = \cdots = x_{-n} , x_{-n}\neq x_{-n -1},\text{with } n \geq 2, \\
\frac{1}{2}(1-r_n)  &\text{ if } g \neq x_{-1} \text{ and } x_{-1} = \cdots = x_{-n} , x_{-n}\neq x_{-n -1},\text{with } n \geq 2, \\
\frac{1}{2}\mathbf{1}\{g \neq x_{-1}\}, &\text{ iff } x_{-1}\neq x_{-2}.
\end{cases} \label{esempio}
\end{equation}

Let us assume that
$$
\sum_{n=2}^{\infty}\prod_{j=2}^nr_j<\infty.
$$
In this example, clearly $\H \subset S.$ We will now explain why, when taking $\H = S$, 
the assumptions required for Algorithm \ref{algo:ps} and  Algorithm \ref{algo:ps2} are not satisfied. In fact, first we have that 
$$
\sum_{g\in \A}\inf\{p(g|x_{-\infty}^{-1}):x_{-\infty}^{-1}\in S\}=0.
$$
Moreover, for any $n\geq 1$, for any $b,g\in \A$ with $g\neq b$,
$$
\inf\{p(g|x_{-\infty}^{-1}):x_{-\infty}^{-1}\in S, x_{-j}=b, \text{ for all } j=1,\ldots,n\} =0.
$$
Therefore, for any $n \geq 1$ the Markov kernel of order $n$
defined in Definition \ref{def:ptildeandpmarkov}  always has the states $\mathbf{0}^{-1}_{-n}$, $\mathbf{1}^{-1}_{-n}$, $\mathbf{2}^{-1}_{-n}$ as absorbing states, which violates the condition of possessing a unique closed class. 

Finally, recalling \eqref{eq:admissiblehistories}, we have that $\H=\{a_{-\infty}^{-1}\in \A^{\N}:a_{-j} \neq a_{-(j-1)}, \text{ for all } j= 2,3,\ldots\}$. Therefore,
$$
\alpha(g|b)=\frac{1}{2},
$$
for any $b\in \A$ and $g\neq b$. This implies that the Markov kernel of order $1$ defined in Definition \ref{def:ptildeandpmarkov} is aperiodic. Therefore, Assumption \ref{as:algo2} holds for this kernel.

Note that by considering $\H$ instead of $S$ does not cause problems when our aim is to sample the chain adapted to the kernel in stationary regime. By the first lemma of Borel-Cantelli, all histories with two equal adjacent symbols are not present in the stationary regime.
\end{example}

\section*{Acknowledgments}

This work was produced as part of the activities of FAPESP Research, Innovation and Dissemination Center for Neuromathematics (grant \# 2013/ 07699-0 , S.Paulo Research Foundation (FAPESP)). 
KL was successively supported by FAPESP fellowship (grant 2022/07386-0 and 2023/12335-9). We thank Mauro Piccioni, Sandro Gallo and Vicenzo Bonasorte for interesting discussions.

\bibliographystyle{apalike}
\bibliography{bib}


\end{document}